\documentclass[12pt]{amsart}
\usepackage{amssymb,mathrsfs}
\usepackage[matrix,arrow,curve]{xy}
\usepackage{epsfig}
\usepackage{setspace}
\usepackage{url}
\newcounter{cequation}[section]
\newtheorem{theorem}[cequation]{Theorem}
\newtheorem*{theorem*}{Theorem}
\newtheorem{lemma}[cequation]{Lemma}
\newtheorem{corollary}[cequation]{Corollary}
\newtheorem{conjecture}[cequation]{Conjecture}
\newtheorem{question}[cequation]{Question}
\newtheorem{proposition}[cequation]{Proposition}

\newtheorem{problem}[cequation]{Problem}
\newtheorem{lemma-definition}[cequation]{Lemma-Definition}

\theoremstyle{definition}
\newtheorem{example}[cequation]{Example}
\newtheorem{definition}[cequation]{Definition}
\newtheorem*{definition*}{Definition}

\theoremstyle{remark}
\newtheorem{remark}[cequation]{Remark}

\makeatletter\@addtoreset{equation}{section}
\@addtoreset{cequation}{section}

\makeatother

\def \WW {\mathcal{W}}

\def \CC {\mathcal{C}}
\def \MM {\mathcal{M}}

\def \QQ {\mathcal{Q}}
\def \P {\mathbb{P}}

\def \C {\mathbb{C}}
\def \TT {\mathbb{T}}
\def \CC {\mathbb{C}}
\def \Z {\mathbb{Z}}
\def \ZZ {\mathbb{Z}}

\def \Aff {\mathbb{A}}

\def \Ar {\mathrm{Ar}}

\def \Ver {\mathrm{Ver}}

\def \In {\mathrm{In}}
\def \wt {\mathrm{wt}}

\def \Res {\mathrm{Res}\,}

\def \G {\mathrm{Gr}}

\newcommand{\pic}{\mathrm{Pic}\,}
\newcommand{\Spec}{\mathrm{Spec}\,}

\newcommand{\arrow}[2]{\langle #1\to #2 \rangle}

\def \ge {\geqslant}
\def \le {\leqslant}

\title{Laurent phenomenon for Landau--Ginzburg models of complete intersections in Grassmannians of planes
}

\author{Victor Przyjalkowski, Constantin Shramov}

\thanks{
This work is supported by the Russian Science Foundation under grant 14-50-00005.
}

\address{
Steklov Mathematical Institute of Russian Academy of Sciences, 8 Gubkina st., Moscow 119991, Russia
}\email{victorprz@mi.ras.ru, costya.shramov@gmail.com}

\begin{document}

\begin{abstract}
In a spirit of Givental's constructions
Batyrev, Ciocan-Fontanine, Kim, and van Straten
suggested Landau--Ginzburg models for smooth Fano complete
intersections in Grassmannians and partial flag varieties
as certain complete intersections in complex tori
equipped with special functions called superpotentials.
We provide a particular algorithm
for constructing birational isomorphisms of these models
for complete intersections in Grassmannians of planes
with complex tori. In this case the superpotentials are
given by Laurent polynomials. 
We study Givental's integrals for
Landau--Ginzburg models suggested by
Batyrev, Ciocan-Fontanine, Kim, and van Straten
and show that they are periods for
pencils of fibers of maps provided by Laurent polynomials we obtain.
The algorithm we provide after minor modifications
can be applied in a more general context.

\end{abstract}

\maketitle

\tableofcontents

\section{Introduction}
Mirror Symmetry declares duality between algebraic and symplectic geometries of different varieties.
Starting from duality between Calabi--Yau varieties it was extended to Fano varieties (see e.\,g.~\cite{Ko94}).
In this case the dual object to a Fano variety is called a Landau--Ginzburg model. Its definition varies
depending on a version of Mirror Symmetry conjectures.
In what follows mirror partners for us are smooth Fano varieties and their toric Landau--Ginzburg models.

The challenge for Mirror Symmetry is to find mirror partners for given varieties and, via studying them, explore
(or check or at least guess) geometry of the initial varieties. In the paper we mostly focus on the problem of finding Landau--Ginzburg models.
Historically constructions of Landau--Ginzburg models, initiated by Givental (\cite{Gi97b}),
Eguchi, Hori and Xiong (\cite{EHX97}),
Batyrev (\cite{Ba97}), Batyrev, Ciocan-Fontanine, Kim and van~Straten (\cite{BCFKS98}), Hori and Vafa~(\cite{HV00}),
and other people, were based on a toric approach. The idea is, given a smooth toric Fano variety or a variety having a (nice) toric degeneration,
to construct a Laurent polynomial whose support is a fan polytope of either a smooth toric variety or a central fiber of a toric degeneration. This idea is initiated by
Batyrev--Borisov approach to treating mirror duality for a toric variety as a classical duality of toric varieties corresponding to dual polytopes.
The constructed Laurent polynomial gives a map from a complex torus to an affine line and, thus, defines a Landau--Ginzburg model.
We do not discuss methods of finding appropriate Laurent polynomial in a general case; see~\cite{Ba97},~\cite{Prz13},~\cite{CCGGK12} for details.

In~\cite{Gi97b} (see discussion after Corollary~0.4 therein) Givental suggested an approach to writing down Landau--Ginzburg models for complete intersections in toric varieties
or varieties having nice (say, terminal Gorenstein) toric degenerations (see also~\mbox{\cite[\S7.2]{HV00}}). This approach assumes an existence of a \emph{nef-partition}
of the set of rays $\mathscr{E}$ of the toric variety's fan.
That is, for each hypersurface that defines the complete intersection a subset of $\mathscr{E}$ should be fixed,
such that the sum of divisors corresponding to rays in the subset is linearly equivalent to this hypersurface,
and all such subsets are disjoint (see Definition~\ref{def:GLG}).

In some cases this construction is described in details
(see e.\,g.~\cite{BCFKS98} for complete intersections in (partial) flag varieties,~\mbox{\cite[\S7.2]{HV00}} for complete intersections,~\cite{DHKLP} for Fano threefolds). A priori the output of the construction is a quasiaffine variety with a complex-valued function (called \emph{superpotential}) on it. However in many cases such quasiaffine varieties are birational to tori and thus
the functions on them (under some additional assumptions) are toric Landau--Ginzburg models (which is very useful for studying them), see, for instance,~\cite{Prz10}.
We can't prove this phenomenon in the general case: there is no approach to construct (or even prove an existence of) a nef-partition.
In some cases, however, the nef-partition is described.
Landau--Ginzburg models for Grassmannians themselves were suggested in~\mbox{\cite[B25]{EHX97}} as Laurent polynomials, i.\,e. functions on tori already.
In ~\cite{BCFKS97} and~\cite{BCFKS98} the description for Landau--Ginzburg models for complete intersections in Grassmannians and partial flag
varieties
as complete intersections in tori are given.
However it is not clear if they themselves are birational to tori (cf.~\cite[Problem 16]{Prz13}).
We prove the phenomenon in the case of complete intersections of Grassmannians of planes.

A basic and the most important property of Landau--Ginzburg models from the point of view discussed above is a \emph{period condition}.
It relates (some of) the periods of a Landau--Ginzburg model for a Fano variety with its \emph{$I$-series}, a generating series for one-pointed
Gromov--Witten invariants. These series for Grassmannians are found in~\cite{BCK03}.
Moreover, Quantum Lefschetz Theorem (see for instance~\cite[Theorem 0.1]{Gi97b},~\cite{Ki00},~\cite{Lee01}) relates
$I$-series of a Fano variety and a complete intersection therein with nef anticanonical class. Thus this series is known for complete intersections in Grassmannians as well.
Periods for these complete intersections are proven to be related with their $I$-series in~\cite{BCFKS98} modulo the assumption that
this holds for Grassmannians themselves. This assumption is proven for Grassmannians
of planes in~\mbox{\cite[Proposition~3.5]{BCK03}} and for all
Grassmannians in~\cite{MR13}. Thus the period condition holds for Landau--Ginzburg models for smooth Fano complete intersections in Grassmannians
suggested in~\cite{BCFKS97}.

\emph{A weak Landau--Ginzburg model} is a Laurent polynomial for which the period condition holds.
The stronger notion, a notion of \emph{toric Landau--Ginzburg model}, requires two more conditions.
\emph{A Calabi--Yau condition} states an existence of a relative compactification
of the family that is a (non-compact) Calabi--Yau variety.
\emph{A toric condition} states that a Newton polytope of the Laurent polynomial is a fan polytope of a toric degeneration of the Fano variety.

\begin{theorem}[{Corollary~\ref{corollary:main}}]
\label{theorem: premain}
Any smooth Fano complete intersection in a Grassmannian of planes has a weak Landau--Ginzburg model.
\end{theorem}

The particular form of this weak Landau--Ginzburg model can be derived from Theorem~\ref{theorem: main}.
In the papers~\cite{Pr16} and~\cite{Psh15b} other algorithms
to transform Landau--Ginzburg models suggested
in~\cite{BCFKS97} to weak Landau--Ginzburg models are presented.
They use completely different approaches inspired by~\cite{CKP14} and gives a shorter way to obtain
Theorem~\ref{theorem: main}.

\begin{conjecture}
\label{conjecture: Grassmannians}
The assertion of Theorem~\ref{theorem: premain}
holds for complete intersections in any Grassmannian or, more generally, partial flag variety.
\end{conjecture}

This conjecture is proved by methods different from ones used in this paper for complete intersections in Grassmannians and
for a big class of complete intersections in partial flag varieties in~\cite{DH15} and for complete intersections in Grassmannians in~\cite{PSh15b}.

Theorem~\ref{theorem: premain} is one more evidence for the following conjecture, that may be regarded as
a strong version of Mirror Symmetry of variations of Hodge structures conjecture, cf.~\cite[Conjecture 38]{Prz13}.

\begin{conjecture}
Any smooth Fano variety has a toric Landau--Ginzburg model. 
\end{conjecture}
%
%
%

Constructions of weak and toric Landau--Ginzburg models give information about toric degenerations
of Fano varieties and enable one to make effective computations.
In addition they can be considered as a first step of our approach to studying Mirror Symmetry.
From the Homological Mirror Symmetry point of view, a Landau--Ginzburg model  is a family of compact varieties over~$\Aff^1$.
A crucial role in its construction is played by singularities of fibers, so compactness of fibers is needed to guarantee
that all singularities we need are ``visible'' on the Landau--Ginzburg model.
A natural way to get a family of compact varieties is to construct
a Calabi--Yau compactification. Compactification Principle (see~\cite[Principle 32]{Prz13}) states that this
compactification gives Landau--Ginzburg models for Homological Mirror Symmetry.
Such Calabi--Yau compactifications are constructed for Fano threefolds (see~\cite{Prz13} and~\cite{Prz16}) and complete
intersections (see~\cite{Prz13},~\cite{PSh13}, and~\cite{Prz18}). Other constructions of relatively compact Landau--Ginzburg models for Grassmannians, not using weak Landau--Ginzburg models, can be found  in~\cite{MR13}.
A challenging problem is to compactify weak Landau--Ginzburg models provided by Theorem~\ref{theorem: premain}.

\medskip

The paper is organized as follows. In Section~\ref{section:toric-LG} we give definitions of toric Landau--Ginzburg models and $I$-series
for Fano varieties.
In Section~\ref{section:toric} we give definitions of Givental's period integrals and Givental's Landau--Ginzburg models for complete intersections
in smooth toric varieties.
They are defined via nef-partitions and relations between rays of a fan defining the toric variety. Also, in Section~\ref{section:toric}
we show how changing variables one can simplify Givental's integrals and Landau--Ginzburg models and get rid of the part depending
on the relations.

Givental's approach can be applied in a more general case. That is one can define Landau--Ginzburg models and period integrals
of smooth Fano complete intersections in Fano varieties having ``good'', say terminal Gorenstein, toric degenerations.
To do this one applies the Givental's method to a certain complete intersection in a crepant resolution of the toric degeneration
of the latter Fano variety and then taking a specializations of some parameters.
Following~\cite{BCFKS98} we show how it works for Grassmannians of planes in Section~\ref{section: ci in Grassmainans}.
In Section~\ref{section: main theorem} we reformulate notions given in Section~\ref{section: ci in Grassmainans}
in a more abstract way suitable for the proof of our main assertions, and formulate Theorem~\ref{theorem: main}
that is a more technical counterpart of Theorem~\ref{theorem: premain}. Sections~\mbox{\ref{section:horizontal}--\ref{section:vertical}}
contain our main technical lemmas needed for the proof of Theorem~\ref{theorem: main}.
The proof itself is given in
Section~\ref{section:proof}.

Section~\ref{section: periods for Grassmannians}
is devoted to further simplifications of Givental's integrals. Theorem~\ref{theorem: main}
states that there is a series of changes
of variables allowing one to birationally
present Landau--Ginzburg models for complete intersections in Grassmannians of planes
as Laurent polynomials. In Section~\ref{section: periods for Grassmannians} we check that these changes of variables
(and also ones for complete intersections in projective spaces) agree with changing Givental's integrals.
This shows that Givental's integrals are indeed periods; they can be easily computed as constant term series.
As a corollary one gets the fact that complete intersections in Grassmannians of planes have weak Landau--Ginzburg models.

In Section~\ref{section:hyperplane-sections} we write down explicit formulas
that are obtained in the proof of Theorem~\ref{theorem: main} for Fano intersections
of Grassmannians of planes with several hyperplanes.
Section~\ref{section:examples} provides a series of examples worked out
as an application of Theorem~\ref{theorem: main} in some other interesting cases.

Due to the lack of space we omit some boring computations in Sections~\ref{section:horizontal},
\ref{section:mixed}, \ref{section:vertical} and~\ref{section:examples}.
In particular we provide detailed proof of technical Lemma~\ref{lemma:horizontal-block-G2-start}
but we skip proofs of Lemmas~\ref{lemma:horizontal-block-G2},~\ref{lemma:horizontal-block-G2-narrow}
,~\ref{lemma:mixed-block-G2},~\ref{lemma:mixed-block-G2-start},~\ref{lemma:vertical-block-G2}
 because they are very similar to the
proof of Lemma~\ref{lemma:horizontal-block-G2-start}. All omitted details can be found in~\cite{PSh-arxiv}.

Those who have a certain amount of combinatorial courage can apply
the approaches described in our paper in more general cases, say, for
complete intersections in arbitrary Grassmannians
or even in partial flag varieties.
We discuss this in Section~\ref{section:discussion}. We also discuss
how one can study weak toric Landau--Ginzburg models given in the proof of
Theorem~\ref{theorem: main} in a deeper way.


\medskip

{\bf Notation and conventions.} Everything is over $\C$.
We use (co-)homology groups with integral coefficients and denote $H^*(X,\ZZ)$ by $H^*(X)$ and  $H_*(X,\ZZ)$ by $H_*(X)$.
Given two integers $n_1$ and $n_2$, we denote the set $\{i\in \ZZ \mid n_1\le i\le n_2\}$ by $[n_1,n_2]$.
Calabi--Yau varieties in this paper are projective varieties with trivial canonical class.
We often use the same notation for a (Cartier) divisor on a variety $X$ and its class in $\pic(X)$.
When we speak about hyperplane or hypersurface sections of a Grassmannian we mean hyperplane or hypersurface sections in its Pl\"ucker embedding.

{\bf Acknowledgments.} The authors are very grateful to B.\,Kim and I.\,Ciocan-Fontanine for explanations of Givental's integrals and for many useful remarks, and also to A.\,Harder, A.\,Fonarev, S.\,Gorchinskiy, and A.\,Kuznetsov for inspiring discussions.

\section{Toric Landau--Ginzburg models}
\label{section:toric-LG}

In this section we define the main objects of our
considerations --- toric Landau--Ginzburg models. For more details and examples see~\cite{Prz13} and
references therein.

Let $X$ be a smooth Fano variety of dimension $N$ and Picard number $\rho$.
Choose a basis
$$
\{H_1,\ldots,H_{\rho}\}$$
in $H^2(X)$
so that for any $i\in [1,\rho]$ and any class $\beta$ in the cone of effective curves
$K$ of $X$ one has~\mbox{$H_i\cdot\beta\ge 0$}.
Introduce formal variables $q^{\sigma_i}$, $i\in [1,\rho]$, and denote $q_i=q^{\sigma_i}$.
For any~\mbox{$\beta\in H_2(X)$} denote
$$q^\beta=q^{\sum \sigma_i (H_i\cdot \beta)}.$$
Consider the Novikov ring $\C_q$, i.\,e. a group ring for $H_2(X)$. We treat it as a ring of polynomials over $\C$ in formal variables
$q^\beta$ with relations
$$q^{\beta_1}q^{\beta_2}=q^{\beta_1+\beta_2}.$$
Note that for any $\beta\in K$ the monomial $q^\beta$ has non-negative degrees in $q_i$.

Let the number
$$
\langle \tau_a \gamma \rangle_\beta, \quad a\in \ZZ_{\geqslant 0},\ \gamma \in H^*(X),\ \beta \in K,
$$
be a one-pointed Gromov--Witten invariant with descendants for $X$, see~\cite[VI-2.1]{Ma99}.
Let $\mathbf 1$ be the fundamental class of $X$.
The series
$$
I^X_{0}(q_1,\ldots,q_{\rho})
=1+\sum_{\beta \in K} \langle\tau_{-K_X\cdot\beta-2} \mathbf 1\rangle_{\beta}
\cdot q^\beta
$$
is called \emph{a constant term of $I$-series} (or \emph{a constant term of Givental's $J$-series}) for $X$
and the series
$$
\widetilde{I}^X_{0}(q_1,\ldots,q_{\rho})=
1+\sum_{\beta \in K} (-K_X\cdot \beta)!\langle\tau_{-K_X\cdot\beta-2} \mathbf 1\rangle_{\beta}
\cdot q^\beta
$$
is called \emph{a constant term of regularized $I$-series} for $X$.
Given a divisor class $H=\sum \alpha_i H_i$
one can restrict these series to a direction corresponding to this divisor setting
$\sigma_i=\alpha_i \sigma$ and~\mbox{$t=q^\sigma$}.
Given a class of symplectic form $[\omega]$ consider a divisor class $D$ associated with it.
We are interested in restriction of the $I$-series to \emph{orbit of the anticanonical direction associated with $\omega$},
so we replace $q^\beta$ by $e^{-D\cdot \beta}t^{-K_X\cdot \beta}$.
In particular one can define a \emph{restriction of a constant term of regularized $I$-series to anticanonical direction} (so $\omega=0$);
it has the form
$$
\widetilde{I}^X_0(t)=1+a_1t+a_2t^2+\ldots,\ \ \ \ a_i\in \CC.
$$

\begin{definition}[{see~\cite[\S6]{Prz13}}]
\label{definition: toric LG}
\emph{A toric Landau--Ginzburg model} of $X$ is a Laurent polynomial $\mbox{$f\in \CC[x_1^{\pm 1}, \ldots, x_N^{\pm 1}]$}$ which satisfies:
\begin{description}
  \item[Period condition] The constant term of $f^i$ equals $a_i$ for any $i$.
  \item[Calabi--Yau condition] There exists a relative compactification of a family
$$f\colon (\CC^*)^N\to \CC$$
whose total space is a (non-compact) smooth Calabi--Yau
variety $LG(X)$. Such compactification is called \emph{a Calabi--Yau compactification}.
  \item[Toric condition] There is a degeneration
   $X\rightsquigarrow T$ to a toric variety~$T$ whose fan polytope
  (i.\,e. the convex hull of generators of its rays) coincides with Newton polytope
  (i.\,e. the convex hull of the support) of $f$.
\end{description}
\end{definition}

\begin{remark}
The period condition is a numerical expression of coincidence of constant term of regularized $I$-series and a period of the family provided by $f$, see Remark~\ref{remark: Picard--Fuchs} and Theorem~\ref{theorem: Picard--Fuchs}.
\end{remark}

Let us remind that the Laurent polynomials for which the period condition is satisfied are called {weak Landau--Ginzburg models};
ones for which in addition a Calabi--Yau condition holds are called \emph{weak Landau--Ginzburg models}.

Toric Landau--Ginzburg models are known for Fano threefolds (see~\cite{Prz13}, \cite{ILP13}, and~\cite{DHKLP}) and
complete intersections in projective spaces (\cite{ILP13}); some other partial results are also known.

There are two usual ways to find toric Landau--Ginzburg models. The first way is to find birational transformations
of known suggestions for Landau--Ginzburg models to make their total spaces tori. In this case their superpotentials
in toric coordinates are Laurent polynomials.
After this one can try to prove the three conditions for being a toric Landau--Ginzburg model.
An important case of this approach is the following. Given a Fano variety one can sometimes describe it
(or its ``good'' degeneration) as a complete intersection in a toric variety.
Then one can try to find a relative birational isomorphism of Givental's type Landau--Ginzburg models (see Definition~\ref{def:GLG}) with a torus.
An example of this approach can be found in~\cite{Prz10} and~\cite{CCGGK12}.

The second way is to find a toric degenerations of $X$. Given this degeneration one has a Newton polytope of its possible toric
Landau--Ginzburg model and so can try to find its particular coefficients. One can look at~\cite{DHKLP} for an example of this approach.

In this paper we apply the first approach to find candidates for toric Landau--Ginzburg models for complete intersections in Grassmannians
using Givental suggestions of Landau--Ginzburg models. We conjecture that Laurent polynomials we get are toric Landau--Ginzburg models.

%
\section{Complete intersections in smooth toric varieties}
\label{section:toric}

In this section we describe Givental's construction of Landau--Ginzburg models and period integrals for complete intersections in toric varieties given in~\cite{Gi97b} (see discussion after Corollary~0.4 therein).
That is, we describe their weak Landau--Ginzburg models and discuss their periods.
Further in Section~\ref{section: ci in Grassmainans} we literally repeat these considerations for complete intersections in singular toric varieties (which are terminal Gorenstein degenerations of Grassmannians).

Let $X$ be a factorial $N$-dimensional toric Fano variety of Picard rank $\rho$ corresponding to
a fan $\Sigma_X$ in a lattice $\mathcal{N}\cong\ZZ^N$. Let $D_1,\ldots, D_{N+\rho}$ be its 
prime invariant divisors.
Let~\mbox{$Y_1,\ldots,Y_l$} be ample divisors in $X$ cutting out a smooth Fano complete intersection
$$Y=Y_1\cap\ldots\cap Y_l.$$
Put $Y_0=-K_X-Y_1-\ldots-Y_l$.
Choose a
basis
$$
\{H_1,\ldots,H_\rho\}\subset H^2(X)$$
so that for any $i\in [1,\rho]$ and any curve $\beta\in K$ of $X$ one has $H_i\cdot\beta\ge 0$.
Introduce variables $q_1,\ldots,q_{\rho}$ as in Section~\ref{section:toric-LG}. Define $\kappa_i$ by $-K_Y=\sum \kappa_i H_i$.

The following theorem is a particular case of Quantum Lefschetz hyperplane theorem, see~\cite[Theorem 0.1]{Gi97b}.
\begin{theorem}
Suppose that $\dim (Y)\geqslant 3$. Then the constant term of regularized $I$-series for $Y$ is given by
\begin{equation}\label{particular-Quantum-Lefschetz}
\widetilde{I}^Y_{0}(q_1,\ldots,q_{\rho})=
\exp\big(\mu(q)\big)\cdot
\sum_{\beta\in K}q^\beta\frac{\prod_{i=0}^l|\beta\cdot Y_i|!
}{\prod_{j=1}^{N+\rho}|\beta\cdot D_j|!^{\frac{\beta\cdot D_j}{|\beta\cdot D_j|}}}
\end{equation}
where $\mu(q)$
is a correction term linear in $q_i$ (in particular it is trivial in the higher index case).
For $\dim(Y)=2$ the same formula holds after replacing $H^2 (Y)$ in the definition of~\mbox{$\widetilde{I}^Y_{0}$} given in Section~\ref{section:toric-LG}
by the restriction of $H^2(X)$ to $Y$.
\end{theorem}

\begin{remark}
Note that the summands of
the series~\eqref{particular-Quantum-Lefschetz} have non-negative degrees in $q_i$.
\end{remark}

Now we describe Givental's construction of a dual Landau--Ginzburg model of $Y$ and compute its periods.
Introduce $N$ formal variables $u_1,\ldots,u_{N+\rho}$
corresponding to divisors~\mbox{$D_1,\ldots,D_{N+\rho}$.}

Let $\mathcal{M}=\mathcal{N}^{\vee}$, and let $\mathcal{D}\cong\ZZ^{N+\rho}$ be a lattice with
a basis $\{D_1,\ldots, D_{N+\rho}\}$ (so that one has a natural identification $\mathcal{D}\cong\mathcal{D}^{\vee}$).
By~\cite[Theorem~4.2.1]{CLS11} one has an exact sequence
$$
0\to \mathcal{M} \to \mathcal{D}
\to A_{N-1}(X)=\pic(X)\cong\ZZ^{\rho}\to 0.
$$
We use factoriality of $X$ here to identify the class group $A_{N-1}(X)$ and the Picard
group~\mbox{$\pic(X)$}.
Dualizing this exact sequence, we obtain
an exact sequence
\begin{equation}
\label{sequence}
0\to \pic(X)^{\vee}\to \mathcal{D}\to \mathcal{N}\to 0.
\end{equation}
Thus $\pic(X)^{\vee}$ can be identified with the lattice of relations on primitive vectors on the
rays of $\Sigma_X$ considered as Laurent monomials in variables $u_i$.
On the other hand, as the basis in $\pic(X)$ is chosen we can identify $\pic(X)^{\vee}$ and $\pic(X)=H^2(X)$.
Hence we can choose a basis in the lattice of relations on primitive vectors on the
rays of $\Sigma_X$ corresponding to $\{H_i\}$ and, thus, to $\{q_i\}$.
We denote these relations by $R_i$, and interpret them as monomials in
the variables~\mbox{$u_1,\ldots,u_{N+\rho}$}.
We also denote by $D_i$ the images of $D_i\in \mathcal D$ in $\pic X$.

Choose a nef-partition, i.\,e.
a partition of the set $[1,N+\rho]$ into sets $E_0,\ldots,E_l$ such that for any $i\in [1,l]$
the divisor $\sum_{j\in E_i} D_j$ is linearly equivalent to $Y_i$ (which also
implies that the divisor~\mbox{$\sum_{j\in E_0} D_j$} is linearly equivalent to $Y_0$).

The following definition is well-known (see discussion after Corollary~0.4 in~\cite{Gi97b},
and also~\cite[\S7.2]{HV00}).

\begin{definition}
\label{def:GLG}
\emph{Givental's Landau--Ginzburg model} for $Y$
is a variety~\mbox{$LG_0(Y)$} in a torus
$$T=\Spec \C_q[u_1^{\pm 1}, \ldots,u_{N+\rho}^{\pm 1}]$$
given by equations
\begin{equation}\label{eq:Rq}
R_i=q_i, \ i\in [1,\rho],
\end{equation}
and
\begin{equation}\label{eq:Eu}
\left(\sum_{s\in E_j} u_s\right)=1, \ j\in [1,l],
\end{equation}
with a function $w=\sum_{s\in E_0} u_s$ called \emph{superpotential}. 
Given 
a symplectic form $\omega$ with~\mbox{$[\omega]\sim \sum \omega_iH_i,$}
where $[\omega]$ is the class in $\pic(Y)$ corresponding to $\omega$,
define \emph{the Givental's Landau--Ginzburg model~\mbox{$LG(Y,\omega)$} associated to $\omega$}
specializing $q_i=\exp(\omega_i)$.
If $\omega$ is an anticanonical form~$\omega_Y$, i.\,e. one has
$[\omega]=-K_Y$, we say for simplicity that~\mbox{$LG(Y)=LG(Y,\omega)$} is an anticanonical Givental's Landau--Ginzburg model for $Y$
instead of saying that~\mbox{$LG(Y,\omega_Y)$} is a Givental's Landau--Ginzburg model for~\mbox{$(Y,\omega_Y)$}.
\end{definition}

One can define a Landau--Ginzburg model associated to a symplectic form in slightly another way, multiplying coefficients
of a divisor corresponding to the form by some number, say $2\pi i$.

\begin{remark}
\label{remark:shift}
The superpotential of Givental's Landau--Ginzburg models can be defined as~\mbox{$w'=u_1+\ldots+u_{N+\rho}$}. However we don't make a distinction between
two superpotentials $w$ and $w'$ as $w'=w+l$, so both these functions define the same family over $\C_q$.
\end{remark}

Given variables $x_1,\ldots, x_r$, define a \emph{standard logarithmic form in these variables} as the form
\begin{equation}
\label{eq: standard form}
\Omega(x_1,\ldots,x_r)=\frac{1}{(2\pi i)^r}\frac{dx_1}{x_1}\wedge\ldots\wedge\frac{dx_r}{x_r}.
\end{equation}

The following definition is well-known (see discussion after Corollary~0.4 in~\cite{Gi97b},
and also~\cite{Gi97a}).

\begin{definition}
\label{def:integral}
Fix $N+\rho$ real positive numbers $\varepsilon_1,\ldots,\varepsilon_{N+\rho}$ and define an $(N+\rho)$-cycle
$$
\delta = \{|u_i=\varepsilon_i|\}\subset \C[u_1^{\pm 1},\ldots, u_{N+\rho}^{\pm 1}]
.$$
\emph{Givental's integral} for $Y$ or $LG_0(Y)$ is an integral
\begin{equation}\label{eq:Givental-integral}
I_Y^0=
\int\limits_{\delta}\frac{
\Omega (u_1,\ldots,u_{N+\rho})}
{\prod_{i=1}^\rho (1-\frac{q_i}{R_i})\cdot \prod_{j=0}^l\left(1-\left(\sum_{s\in E_j} u_s\right)\right)}\in \C[[q_1,\ldots,q_\rho]].
\end{equation}
Given a class of symplectic form $\omega$ and a divisor class $D=\sum \omega_i H_i$ associated with it 
one can \emph{specialize Givental's integral to the anticanonical direction and the form $\omega$}
putting~\mbox{$q_i=e^{\omega_i}t^{\kappa_i}$} in the integral~\eqref{eq:Givental-integral}. We denote the result of specialization by $I_{(Y,\omega)}$, and
we put $I_{(Y,\omega)}=I_{Y}$ if $[\omega]=0$, which means that we put $D=0$, so that $w_i=0$ for all~$i$.
\end{definition}

\begin{remark}
The integral~\eqref{eq:Givental-integral} does not depend on numbers $\varepsilon_i$ provided they are small enough.
\end{remark}

\begin{remark}
The integral~\eqref{eq:Givental-integral} is defined up to a sign as we do not specify an order of variables.
\end{remark}

The following assertion is well-known to experts (see~\cite[Theorem~0.1]{Gi97b},
and also discussion after Corollary~0.4 in~\cite{Gi97b}).

\begin{theorem}
\label{theorem:series}
One has
$$
\widetilde{I}^Y_0=I_Y^0.
$$
\end{theorem}

The recipe for Givental's Landau--Ginzburg model and integral can be written down in another, more simple, way.
That is, we make suitable monomial change of variables~\mbox{$u_1,\ldots,u_{N+\rho}$} an get rid
of some of them using equations~\eqref{eq:Rq}. More precisely,
as~$\mathcal{N}$ is a free group, using the exact sequence~\eqref{sequence}
one obtains an isomorphism
$$\mathcal{D}\cong \pic{X}^\vee\oplus \mathcal{N}.$$
Thus one can find a monomial change of variables
$u_1,\ldots, u_{N+\rho}$ to some new variables~\mbox{$x_1,\ldots,x_N, y_1,\ldots,y_\rho$}, so that
$$u_i=\widetilde{X}_i(x_1,\ldots, x_N, y_1,\ldots,y_\rho,q_1,\ldots,q_\rho)$$
such that for any $i\in [1,\rho]$
one has
$$\frac{R_i(u_1,\ldots,u_{N+\rho})}{q_i}=\frac{1}{y_i}.$$
Put
$$X_i=\widetilde{X}_i(x_1,\ldots, x_N, 1,\ldots,1,q_1,\ldots,q_\rho).$$
Then $LG(Y)$ is given in the torus $\Spec \C_q[x_1^{\pm 1},\ldots,x_N^{\pm 1}]$
by equations
$$
\left(\sum_{s\in E_j} \alpha_s X_s\right)=1, \ j\in [1,l],
$$
with superpotential $w=\sum_{s\in E_0}\alpha_s X_s$, where
$\alpha_i=\prod q_j^{r_{i,j}}$ for some integers~$r_{i,j}$.

Let us mention that given
a Laurent monomial $U_i$ in variables $u_j$, $j\in [1,N+\rho]$,
that does not depend on a variable $u_i$ one has
\begin{equation}
\label{equation: integral changes}
\Omega(u_1,\ldots, u_i^{\pm 1}\cdot U_i,\ldots, u_{N+\rho})=\pm\Omega(u_1,\ldots, u_i,\ldots, u_{N+\rho}).
\end{equation}

This means that
\begin{equation}
\label{eq:integral xy}
I_Y^0=
\int_{\delta'}
\frac{
\pm \Omega(y_1,\ldots, y_\rho)\wedge\Omega(x_1,\ldots,x_N)
}{\prod_{i=1}^\rho (1-y_i)\prod_{j=0}^l\left(1-\left(\sum_{s\in E_j} \alpha_s\widetilde{X}_s\right)\right)}
\end{equation}
for some $(N+\rho)$-cycle $\delta'$.

Consider an integral
$$
\int_{\sigma} \frac{dU}{U}\wedge \Omega_0
$$
for some form $\Omega_0$ and a cycle $\sigma=\sigma'\cap \{|U|=\varepsilon\}$
for some cycle~\mbox{$\sigma'\subset \{U=0\}$}.
It is well known
(see, for instance,~\cite[Theorem 1.1]{ATY85}) that
$$
\frac{1}{2\pi i}\int_{\sigma} \frac{dU}{U}\wedge \Omega_0=\int_{\sigma'}\left.\Omega_0\right|_{U=0}
$$
if both integrals are well defined (in particular the form $\Omega_0$ does not have a pole along~\mbox{$\{U=0\}$}).

We denote $$\left.\Omega_0\right|_{U=0}=\Res_U\left(\frac{dU}{U}\wedge \Omega_0\right).$$

Taking residues of the form on the right hand side of~\eqref{eq:integral xy} with respect to $y_i$ one gets
$$
I_Y^0=
\int_{\delta''}
\frac{
\pm \Omega(x_1,\ldots,x_N)
}{\prod_{j=0}^l\left(1-\left(\sum_{s\in E_j} \alpha_s X_s\right)\right)}
$$
for some $N$-cycle $\delta''$.

Moreover, one can introduce a new parameter $t$ and scale $u_i\to t u_i$ for $i\in E_0$.
Fix 
a class of symplectic form $\omega$ and a divisor class $D=\sum \omega_i H_i$ associated with it.
One can check that after a change
of coordinates $q_i=e^{\omega_i}t^{\kappa_i}$ 
the initial integral
restricts to the integral
\begin{equation}
\label{eq:restricted integral}
\int_{\delta_1}
\frac{
\pm \Omega (x_1,\ldots,x_N)
}{\prod_{j=1}^l\left(1-\left(\sum_{s\in E_j} \gamma_s X_s\right)\right)\cdot \left(1-t\left(\sum_{i\in E_0} \gamma_i X_i\right)\right)}=I_{(Y,\omega)}
\end{equation}
for some monomials $\gamma_i$ in $e^{\omega_j}$
and
for some $N$-cycle $\delta_1$ homologous to a cycle
$$\delta_1^0=\{|x_i|=\varepsilon_i\mid i\in [1,N]\}.$$
In particular, for $\omega=0$ one has $\gamma_i=1$.
The same specialization defines the anticanonical Givental's Landau--Ginzburg model of~$Y$, which is given by equations
$$
\left(\sum_{s\in E_j} X_s\right)=1, \quad j\in [1,l],
$$
with superpotential $w=\sum_{s\in E_0} X_s$.

Consider a non-toric variety $X$ that has a small (that is, terminal Gorenstein) toric degeneration $T$.
Let $Y$ be a Fano complete intersection in $X$. Consider a nef-partition for the set of rays of the fan of $T$ corresponding to (degenerations of)
hypersurfaces cutting out $Y$. Let $LG(Y)$ be a result of applying the procedure discussed above
for Givental's Landau--Ginzburg model defined for $T$ and the nef-partition in the same way as in the case of complete
intersections in smooth toric varieties.
Batyrev in~\cite{Ba97} suggested $LG(Y)$ as a Landau--Ginzburg model for $Y$.
Moreover, at least in some cases Givental's integral and Landau--Ginzburg model (associated to anticanonical class) can be simplified further by making birational
changes of variables and taking residues. Thus Givental's Landau--Ginzburg models give weak ones after such transformations.
In Section~\ref{section: periods for Grassmannians} we demonstrate both of these ideas for complete
intersections in projective spaces or Grassmannians of planes.

\section{Complete intersections in Grassmannians}
\label{section: ci in Grassmainans}
The picture described in Section~\ref{section:toric}
can be generalized to complete intersections in Grassmannians. The difference
is that Grassmannians are not toric. However they have \emph{small toric degenerations}, i.\,e. degenerations to terminal Gorenstein
toric varieties, see~\cite{St93}. The mirror construction for complete intersections in Grassmannians can be derived
from crepant resolutions of these degenerations. In this section we describe some constructions from~\cite{BCFKS97} and~\cite{BCFKS98}
for a Grassmannian $G=\G(n, k+n)$.

Fix two integers $n$ and $k$ such that $n, k\ge 2$.
We define a quiver $\QQ_0$
as a set of vertices
$$
\Ver(\QQ_0)=\{(i,j)\mid i\in [1,k], j\in [1,n]\}\cup\{(0,1), (k,n+1)\}
$$
and a set of arrows
$\Ar(\QQ_0)$ described as follows.
All arrows are either \emph{vertical} or \emph{horizontal}.
For any $i\in [1,k-1]$ and any $j\in [1,n]$ there is one
vertical arrow $\arrow{(i,j)}{(i+1,j)}$ that goes from the vertex $(i,j)$
down to the vertex $(i+1,j)$. For any $i\in [1,k]$ and any $j\in [1,n-1]$
there is one horizontal arrow
$\arrow{(i,j)}{(i,j+1)}$ that goes from the vertex $(i,j)$ to the right
to the vertex $(i,j+1)$.
We
also add an extra vertical arrow $\arrow{(0,1)}{(1,1)}$
and an extra horizontal
arrow  $\arrow{(k,n)}{(k,n+1)}$ to $\Ar(\QQ_0)$, see Figure~\ref{figure:quiverG25}.

\begin{figure}[htbp]
\begin{center}
\includegraphics[width=5cm]{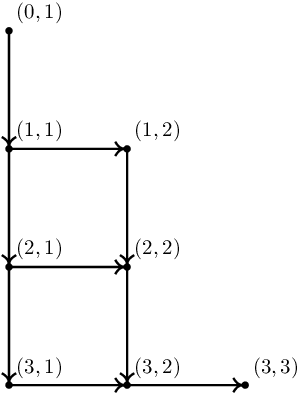}
\end{center}
\caption{Quiver $\QQ_0$ for the Grassmannian $\G(2,5)$}
\label{figure:quiverG25}
\end{figure}

Now we describe a toric degeneration
$P=P(n,k+n)$ of $G$ in its Pl\"ucker embedding.
The arrows of $\QQ_0$ correspond to rays of a fan $\Sigma_P$ of $P$, so we identify them;
relations for the primitive vectors on the rays of $\Sigma_P$
correspond to cycles in $\QQ_0$ if we identify vertices $(0,1)$ and $(k, n+1)$.
The cones of $\Sigma_P$ of dimension at least $2$ are cones over faces of a convex hull of generators of rays of $\Sigma_P$.
A degeneration $P$ is a Fano
toric variety corresponding to $\Sigma_P$.

The variety $P$ is not smooth but terminal Gorenstein.
It admits (some) crepant resolution
that we denote by $\widetilde{P}$.
All relations on rays of $P$ (or $\widetilde{P}$)
are combinations of basic ones described as follows.
For any $i\in [1,k-1]$ and $j\in [1,n-1]$ we have a \emph{box relation}
\begin{multline*}
\arrow{(i,j)}{(i+1,j)}+\arrow{(i+1,j)}{(i+1,j+1)}=\\=
\arrow{(i,j)}{(i,j+1)}+\arrow{(i,j+1)}{(i+1,j+1)};
\end{multline*}
besides that, we have one \emph{roof relation}
\begin{multline*}
0=
\arrow{(0,1)}{(1,1)}+\arrow{(1,1)}{(1,2)}+\ldots+\arrow{(1,n-1)}{(1,n)}+\\
+\arrow{(1,n)}{(2,n)}+\ldots+\arrow{(k-1,n)}{(k,n)}+\arrow{(k,n)}{(k,n+1)},
\end{multline*}
see Figure~\ref{figure:relations}.
These relations, considered as elements of the Picard group of $\widetilde{P}$,
form a basis in it.
The roof relation is a pull-back to $\widetilde{P}$ of
a generator of the Picard group of~$P$.
We introduce variables $q_i$, $i\in [1,(k-1)(n-1)]$, corresponding to
box relations, and
a variable $q$ corresponding to the roof relation.

\begin{figure}[htbp]
\begin{center}
\includegraphics[width=5cm]{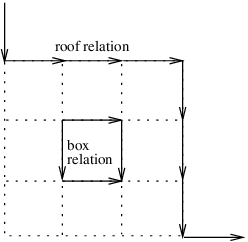}
\end{center}
\caption{Relations} \label{figure:relations}
\end{figure}

Now we describe a nef-partition corresponding to a complete intersection in
the Grassmannian $G$.
For a fixed $s\in [1,k-1]$ the \emph{$s$-th horizontal basic block} is a set of all
arrows~\mbox{$\arrow{(s,j)}{(s+1,j)}$} with $j\in [1,n]$. Similarly, for a fixed $s\in [1,n-1]$ the \emph{$s$-th vertical basic block} is a set of all arrows $\arrow{(i,s)}{(i,s+1)}$ with $i\in [1,k]$. We also define the $0$-th
horizontal basic block as the set that consists of a single
arrow $\arrow{(0,1)}{(1,1)}$, and we define the $n$-th
vertical basic block as the set that consists of a single
arrow $\arrow{(k,n)}{(k,n+1)}$.

A sum of divisors in $P$ associated to rays
corresponding to arrows in any horizontal or vertical basic block
is linearly
equivalent to
a generator of the
Picard group of $P$, see~\cite[Proposition~4.1.4]{BCFKS97}.
Thus given a complete intersection in $G$ one can choose
a nef-partition
that consists of collections of rays corresponding to arrows
of appropriate numbers of vertical or horizontal basic blocks.

The constant term of $I$-series
 of $\widetilde{P}$ is
 $$I=I^0_{\widetilde{P}}(q,q_1,\ldots, q_{(k-1)(n-1)}).$$
In~\cite[Conjecture 5.2.3]{BCFKS97} it was conjectured that
$$\widetilde{I}_0^{G}(q)=I(q,1,\ldots,1).$$
This is proved for $n=2$ in~\cite[Proposition 3.5]{BCK03} and for any $n\ge 2$
in~\cite{MR13}.

Consider a smooth Fano complete intersection $Y$ in $G$.
Let $LG_0(Y)$ be a Givental's Landau--Ginzburg model constructed for
$\widetilde{P}$ and a nef-partition associated
$Y$. Denote it's Givental's integral by $I^0_{Y}$.
In discussion after Conjecture~5.2.1 in~\cite{BCFKS98} it is explained
that, assuming the latter assertion,
one has
$$
\widetilde{I}^Y_0=I^0_{Y},
$$
which can be viewed as an analog of Theorem~\ref{theorem:series} in this particular non-toric case.

Further in Section~\ref{section: main theorem} we will study
the case of complete intersections in a Grassmannian of planes.
For this case there is an explicit formula for constant term of regularized $I$-series (and thus for Givental's integral).
Let
$$
\gamma(r)=\sum_{i\in [1,r]} \frac{1}{i}.
$$

\begin{theorem}[{\cite[Proposition 3.5]{BCK03}}]
\label{theorem: I-series}
Let
$$Y=\G(2,k+2)\cap Y_1\cap \ldots \cap Y_l$$
be a smooth Fano complete intersection
with $\deg Y_i=d_i$, $\sum d_i<k+2$. Denote
$$d_0=k+2-\sum d_i.$$
Then
$$
\widetilde{I}^Y_0=\sum_{d\geqslant 0}
\frac{\prod_{i=0}^l (d_0d_i)!}{d!^{k+2}}\cdot \frac{(-1)^d}{2} \cdot \sum_{r=0}^d
\binom{d}{r}^{k+2}\big( (k+2)\left(d-2r\right)\left( \gamma(r)-\gamma(d-r)\right)-2\big)
\cdot t^{d_0d}.
$$
\end{theorem}

Summarizing, one can deal with a Grassmannian and a complete intersection therein
just replacing the Grassmannian by its small toric degeneration
and applying Givental's procedure to it.

Now we write down explicitly this picture after getting
rid of relations as it is described in
Section~\ref{section:toric}.
The superpotential for $G$ itself is the
polynomial
$$
a_{1,1} + \sum_{\substack{i\in [1,k-1],\\ j\in [1,n]}}
\frac{a_{i+1,j}}{a_{i,j}}+
\sum_{\substack{i\in [1,k],\\ j\in [1,n-1]}}
\frac{a_{i,j+1}}{a_{i,j}}+
\frac{1}{a_{k,n}}
$$
in variables $a_{i,j}$, $i\in[1,k]$, $j\in [1,n]$, see~\cite[B25]{EHX97}.

Consider the following Laurent polynomials:
\begin{equation}\label{eq:monomials}
\begin{aligned}
&T_1=a_{1,1},
\\
&T_{i+1}=\sum_{j\in [1,n]}\frac{a_{i+1,j}}{a_{i,j}},
\quad i\in [1,k-1],
\\
&T_{k+j}=\sum_{i\in [1,k]}\frac{a_{i,j+1}}{a_{i,j}},
\quad j\in [1,n-1],
\\
&T_{k+n}=\frac{1}{a_{k,n}}.
\end{aligned}
\end{equation}
For any arrow
$$\alpha=\arrow{(i,j)}{(i',j')}\in\Ar(\QQ_0)$$
we define $h(\alpha)$ and $t(\alpha)$
as the vertices $(i,j)$ and $(i',j')$, respectively.
One can see that Laurent monomials appearing in~\eqref{eq:monomials}
are of the form $a_{h(\alpha)}/a_{t(\alpha)}$ for some $\alpha\in\Ar(\QQ_0)$,
and Laurent polynomials listed in~\eqref{eq:monomials}
are of the form
$$
\sum\limits_{\alpha\in B}\frac{a_{h(\alpha)}}{a_{t(\alpha)}},
$$
where $B\subset\Ar(\QQ_0)$ is some basic block.

Consider a smooth Fano complete intersection
$$Y=G\cap Y_1\cap \ldots \cap Y_l$$
with $\deg (Y_p)=d_p$.
Choose a splitting $[1,k+n]=E_0\sqcup E_1\sqcup\ldots\sqcup E_l$
with $|E_p|=d_p$, $p\in [1,l]$, so that $|E_0|=k+n-\sum d_p$.
Define $\Sigma_p=\sum_{i\in E_p} T_i$, $p\in [0,l]$.
Then the equations of anticanonical Givental's Landau--Ginzburg model for $Y$ are
\begin{equation}
\label{eq:sigma}
\Sigma_p=1, \quad p\in [1,l],
\end{equation}
and the superpotential is~$\Sigma_0$.

\section{Main theorem}
\label{section: main theorem}

Now
we choose a specific nef-partition we are going to use in our main theorem, i.\,e.
in Theorem~\ref{theorem: main} below.
Informally, for any hypersurface we
take a union of a suitable number of consecutive basic blocks.
To make it more precise we introduce some additional terminology.

A \emph{horizontal block} of size $d$ is a union of
$d$ consecutive basic horizontal blocks.
A \emph{vertical block} of size $d$ is a union of
$d$ consecutive basic vertical blocks.
A \emph{mixed block} of size $d$ is a union of
$d_1$ consecutive basic horizontal blocks including
the $(k-1)$-th one and $d_2$ consecutive basic vertical
blocks including the first one,
where $d_1+d_2=d$.
By a block we will mean either a horizontal block, or a vertical block,
or a mixed block.

\begin{figure}[htbp]
\begin{center}
\includegraphics[width=10cm]{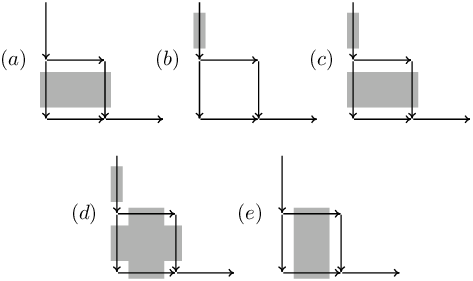}
\end{center}
\caption{Blocks for the Grassmannian $\G(2,4)$}
\label{figure:blocks}
\end{figure}

\begin{example}
Figure~\ref{figure:blocks} represents several examples of blocks
in a quiver corresponding to the Grassmannian~\mbox{$\G(2,4)$}. Namely,
Figures~\ref{figure:blocks}(a), \ref{figure:blocks}(b) and~\ref{figure:blocks}(c)
represent horizontal blocks, that are
the first basic horizontal block, the $0$-th basic horizontal block
and a horizontal block of size~$2$, respectively.
Figure~\ref{figure:blocks}(d) represents a mixed block of size~$3$.
Finally, Figure~\ref{figure:blocks}(e) represents the first basic vertical block.
\end{example}

The set of vertices of a block $B$ is the set
$\Ver(B)\subset\Ver(\QQ_0)$ such that for any~\mbox{$v\in\Ver(B)$}
there is an arrow $\alpha\in B$ with either $t(\alpha)=v$
or $h(\alpha)=v$.
We say that an arrow~\mbox{$\alpha\in\Ar(\QQ_0)$}
is an \emph{inner arrow} of a block $B$,
if~\mbox{$t(\alpha)\in\Ver(B)$} and~\mbox{$h(\alpha)\in\Ver(B)$},
while~\mbox{$\alpha\not\in B$}. We denote the set of inner arrows
for~$B$ by~\mbox{$\In(B)$}.

An \emph{admissible quiver} $\QQ$ is a subquiver
of $\QQ_0$ with a set of vertices $\Ver(\QQ)=\Ver(\QQ_0)$,
and a non-empty set of arrows $\Ar(\QQ)=\Ar(\QQ_0)\setminus B$,
where $B$ is either a horizontal
or a mixed block, and $B$ contains the arrow~$\arrow{(0,1)}{(1,1)}$.
In particular, if $\QQ$ is an admissible quiver
and $B'\subset\Ar(\QQ_0)$ is a block, then
$$B''=B'\cap\Ar(\QQ)$$
is again a block. Note also that if $\QQ$ is an admissible quiver such that
$\Ar(\QQ)$ contains the arrow $\arrow{(0,1)}{(1,1)}$,
then $\QQ=\QQ_0$.

Let $V=\{x_1,\ldots,x_N\}$ be a finite set. We denote the torus
$$\Spec \C [x_1^{\pm 1},\ldots, x_N^{\pm 1}]\cong (\C^*)^N$$
by $\TT(V)$. Note that $x_1,\ldots, x_N$
may be interpreted as coordinates on $\TT(V)$.

A \emph{triplet} is a collection
$(\QQ, V, R)$, where $\QQ$ is an admissible
quiver, $V$ is a finite set of variables, $R$ is a map from the set $\Ver(\QQ)$
to the set
of rational functions in the variables of~$V$.

A rational function associated to a triplet $(\QQ,V,R)$
and a non-empty subset $C\subset\Ar(\QQ)$
is the rational function in the variables of $V$ defined as
$$
F_{\QQ,V,R,C}=\sum\limits_{\alpha\in C}
\frac{R\left(h(\alpha)\right)}{R\left(t(\alpha)\right)}.
$$
A hypersurface $H_{\QQ,V,R,C}\subset \TT(V)$
associated to $(\QQ,V,R)$
and $C$ is defined by the equation
$F_{\QQ,V,R,C}=1$.
A rational function associated to a triplet $(\QQ,V,R)$
is the rational function in the variables of $V$ defined as
$$
F_{\QQ,V,R}=F_{\QQ,V,R,\Ar(\QQ)}=\sum\limits_{\alpha\in\Ar(\QQ)}
\frac{R\left(h(\alpha)\right)}{R\left(t(\alpha)\right)}.
$$

\emph{A change of variables that agrees
with a triplet} $(\QQ,V,R)$
and with a block $B$ is
a rational map
$$\psi\colon \TT(V')\dasharrow \TT(V),$$
where $V'$ is a set of variables such that $|V'|=|V|-1$,
the closure of the image of $\TT(V')$ with respect to $\psi$
 is the hypersurface
$$H_{\QQ,V,R,B}\subset \TT(V),$$
and $\psi$ gives a birational map
between~$\TT(V')$ and $H_{\QQ,V,R,B}$. By a small abuse of terminology
we will sometimes omit either a triplet or a block when speaking
about a change of variables that agrees with something. We will sometimes also
refer to automorphisms of tori as changes of variables, but in such situations
we will not mention any triplets or blocks.

Let $\psi$ be a
change of variables that agrees with a triplet $(\QQ,V,R)$ and with a
block~$B$.
\emph{A transformation of a triplet $(\QQ,V,R)$ associated
to} $\psi$ is a triplet $(\QQ',V',R')$, where~\mbox{$\QQ'\subset\QQ$}
is a quiver with $\Ver(\QQ')=\Ver(\QQ)$
and
$$\Ar(\QQ')=\Ar(\QQ)\setminus B,$$
the set $V'$ is a set of variables
such that $|V'|=|V|-1$, and $R'(i,j)=\psi^*R(i,j)$.

\begin{remark}
Let $(\QQ,V,R)$ be a triplet, $B$ be a block, and $\psi$ be a
change of variables that agrees with the triplet $(\QQ,V,R)$ and with the
block $B$.
Let $(\QQ', V', R')$ be a transformation of a triplet $(\QQ,V,R)$ associated
to $\psi$. Then
$\psi^*F_{\QQ,V,R,B}=1$ and
$$\psi^*F_{\QQ,V,R}=F_{\QQ',V',R'}+1.$$
\end{remark}

Now we can reformulate the description of
Landau--Ginzburg models for complete intersections in Grassmannians discussed in
Section~\ref{section: ci in Grassmainans}
in terms introduced above. Let
$$V_0=\{a_{i,j}\}, \quad i\in [1,k], j\in [1,n].$$
Put $R_0(i,j)=a_{i,j}$ for $i\in [1,k]$, $j\in [1,n]$, and $R_0(0,1)=R_0(k,n+1)=1$.
Let
$$Y=G\cap Y_1\cap \ldots \cap Y_l$$
be a smooth Fano complete intersection.
Let $B_1,\ldots,B_l$ be disjoint horizontal, mixed or vertical
blocks such that $B_i$, $i\in [1,l]$, is a block of size $\deg Y_i$. Put
$$C=\Ar(\QQ_0)\setminus \left(\cup_{i\in [1,l]} B_i\right).$$
Then
a variety that is a complete intersection of hypersurfaces
$H_{\QQ_0,V_0,R_0,B_i}$, $i\in [1,l]$,
in $\TT(V_0)$
equipped with a function $F_{\QQ_0,V_0,R_0,C}$ as superpotential
is a Landau--Ginzburg model of~$Y$ suggested in~\cite{BCFKS98}, cf. equations~\eqref{eq:sigma}.
Theorem~\ref{theorem: main}
states that for given $d_1,\ldots, d_l$ there is a choice of blocks
$B_1,\ldots, B_l$ and a sequence of~$l$
changes of variables such that the Landau--Ginzburg model in is fact birational to a torus,
and a birational equivalence can be chosen so that the
superpotential becomes a Laurent polynomial on this torus.


\begin{theorem}
\label{theorem: main}
Let $n=2$.
Consider the triplet $(\QQ_0,V_0,R_0)$.
Let $d_1, \ldots, d_l$ be positive integers such that
for some $i_0\in [0,l]$ one has $d_1,\ldots,d_{i_0}>1$ and
$d_{i_0+1}=\ldots=d_l=1$. Suppose that
$$\sum d_i<k+2.$$
Then there exist blocks $B_1,\ldots, B_l$, a sequence of
triplets
$$(\QQ_i,V_i,R_i),\quad i\in [1,l],$$
and a sequence of changes of variables
$$\psi_i\colon \TT(V_i)\dasharrow \TT(V_{i-1}),\quad i\in [1,l],$$ such that
\begin{itemize}
\item
the size of the block $B_i$ is $d_i$;
\item one has $B_i\cap B_j=\varnothing$ for $i\neq j$, $i,j\in [1,l]$;
\item
the change of variables
$\psi_i$ agrees with the triplet $(\QQ_{i-1},V_{i-1},R_{i-1})$
and the block~$B_i$;
\item
the triplet $(\QQ_i,V_i,R_i)$ is a transformation of the triplet
$(\QQ_{i-1},V_{i-1},R_{i-1})$ associated to $\psi_i$;
\item
the rational function
$$
F_{\QQ_i,V_i,R_i}=
(\psi_i\circ\ldots\circ \psi_1)^*F_{\QQ_0,V_0,R_0}
$$
is a Laurent polynomial in variables of $V_i$.
\end{itemize}
In particular, the rational function $F_{\QQ_l,V_l,R_l}$
is a Laurent polynomial in $2k-l$ variables.
\end{theorem}

We will prove Theorem~\ref{theorem: main} in Section~\ref{section:proof}.
In order to do this we will deal separately with
changes of variables that agree with horizontal, mixed and vertical
blocks in Sections~\ref{section:horizontal},~\ref{section:mixed}
and~\ref{section:vertical}, respectively.
In most of the cases (except for a relatively easy
Lemma~\ref{lemma:horizontal-block-G2-narrow})
a change of variables will be performed in two steps.
First we will choose some variable
(which we will later refer to as \emph{weight variable}),
and make a monomial change of coordinates multiplying
each variable by a suitably chosen power of the weight variable.
After this
we will exclude another variable (which we will later
refer to as \emph{main variable}) using the equation
of the hypersurface associated to
the triplet and the block, and check that after the corresponding substitution the Laurent polynomial associated to the triplet remains a Laurent
polynomial. One effect that still looks surprising to us
is that the case of a horizontal block of size~$1$ (i.\,e. of a basic horizontal
block) is treated differently from the case of a horizontal block of
size at least~$2$, so that the assertions of Lemmas~\ref{lemma:horizontal-block-G2-narrow}
and~\ref{lemma:horizontal-block-G2} appear to be different indeed.
Finally, since the proofs of the lemmas in Sections~\ref{section:horizontal},
\ref{section:mixed} and~\ref{section:vertical} look rather messy, we
illustrate them in Sections~\ref{section:hyperplane-sections}
and~\ref{section:examples} by a large (and hopefully
representative) sample of examples; we suspect that this may be more instructive
than reading the proofs themselves.

\section{Horizontal blocks}
\label{section:horizontal}

In
this section we write down changes of variables that agree with
horizontal blocks for Grassmannians $\G(2,k+2)$.

\begin{figure}[htbp]
\begin{center}
\includegraphics[width=4.42cm]{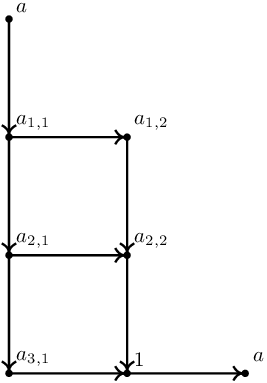}
\end{center}
\caption{Starting triplet for the Grassmannian $\G(2,5)$}
\label{figure:g25one}
\end{figure}

We start with introducing some additional auxiliary notions.

\begin{definition}\label{definition:lambda-degree}
Let $V$ be some collection of variables. Let $W\subset V$ be a subset,
and~\mbox{$\Lambda\colon W\to\Z$} be an arbitrary function.
Let $\mu$ be a Laurent monomial in the variables of $V$.
We define the \emph{$\Lambda$-degree} of
$\mu$ as
$$\deg_{\Lambda}(\mu)=\sum_{a\in W} \Lambda(a)\cdot \deg_{a}(\mu).$$
By a \emph{total degree} of $\mu$ with respect to the variables of
$W$ we mean the $\Lambda$-degree of $\mu$ for the function $\Lambda\equiv 1$,
i.\,e. the sum of degrees of $\mu$ with respect to the variables
of $W$.
\end{definition}

\begin{definition}\label{definition:00s}
Let $s\in [1,k]$. Put
$$W_{\varnothing, \varnothing, s}=\{a_{i,j}\mid i\in [1,s], j\in [1,2]\}$$
if $s<k$, and put
$$W_{\varnothing, \varnothing, k}=\{a_{i,1}\mid i\in [1,k]\}\cup\{a_{i,2}\mid i\in [1,k-1]\}.$$
We define
$$\Lambda_{\varnothing, \varnothing, s}\colon
W_{\varnothing, \varnothing, s}\to\Z$$
as $\Lambda_{\varnothing, \varnothing, s}(a_{i,2})=i-s$ for $i\in [1,s]$
and $\Lambda_{\varnothing, \varnothing, s}(a_{i,1})=i-s+1$ for $i\in [1,s]$,
$i\neq s-1$. Finally, we put~\mbox{$\Lambda_{\varnothing, \varnothing, s}(a_{s-1,1})=1$}.
\end{definition}

Now we start to describe our changes of variables.

\begin{lemma}
\label{lemma:horizontal-block-G2-start}
Suppose that $n=2$.
Let $(\QQ, V, R)$ be a triplet, and $B\subset\Ar(\QQ)$ be a horizontal
block such that the arrow $\arrow{(0,1)}{(1,1)}$ is contained in $B$.
Suppose that $V$ is a set of variables
$$V=\{a_{i,1}\mid i\in [1,k]\}\cup \{a_{i,2}\mid i\in [1,k-1]\}\cup\{a\},$$
and the following conditions hold:
\begin{itemize}
\item[(i)]
one has $R(k,2)=1$;
\item[(ii)]
$R(0,1)=R(k,3)=a$;
\item[(iii)]
for $i\in [1,k]$, $j\in [1,2]$, $(i,j)\neq (k,2)$ one has
$R(i,j)=a_{i,j}$, see Figure~\ref{figure:g25one}.
\end{itemize}

Then there exists a
change of variables $\psi$ that agrees with the triplet $(\QQ,V,R)$
and with the
block $B$
with the following properties. Let $(\QQ'',V'',R'')$
be the transformation of the triplet
$(\QQ,V,R)$ associated to $\psi$,
and let $s$ be the largest number
such that $(s,1)\in\Ver(B)$.
We can assume that
$V''$ is a set of variables
$$V''=\{a''_{i,1}\mid i\in [1,k]\}\cup
\{a''_{i,2}\mid i\in [1,k-1]\}.$$
Then $\psi^*F_{\QQ,V,R}$ is a Laurent polynomial
in the variables of $V''$
and the following assertions hold:
\begin{itemize}
\item[(I)]
the quiver $\QQ''$ does not
contain vertical arrows $\alpha$ such that $h(\alpha)=(i,j)$ for
$i\in [1,s]$, $j\in [1,2]$;
\item[(II)]
one has $R''(k,2)=1$;
\item[(III)]
for $(i,j)$ with $i\in [s,k]$,
$j\in [1,2]$, $(i,j)\neq (k,2)$,
one has $R''(i,j)=a''_{i,j}$;
\item[(IV)] for any $i\in [1,s-2]$ one has
$R''(i,1)=a''_{i,1}\cdot\bar{R}''(i)$;
\item[(V)] one has
$R''(s-1,1)=\bar{R}''(i)$;
\item[(VI)] for any $i\in [1,s-1]$ one has
$R''(i,2)=a''_{i,2}\cdot\bar{R}''(i)$;
\item[(VII)] the rational function $R''(k,3)$ is a Laurent
polynomial in the variables of $V''$ such that
$R''(k,3)$ does not depend on variables $a''_{i,j}$ with $i\in [s+1, k]$,
$j\in [1,2]$, and each of its Laurent monomials
has non-negative degree in each
of the variables~\mbox{$a''_{s,j}$, $j\in [1,2]$};
\item[(VIII)] if $s<k$, then the total degree of any Laurent
monomial of $R''(k,3)$
with respect to variables $a_{i,2}''$, $i\in [1,s]$, is non-positive;
if $s=k$, then the total degree of any Laurent monomial of $R''(k,3)$
with respect to variables $a_{i,2}''$, $i\in [1,k-1]$, is non-positive;
\item[(IX)] the $\Lambda_{\varnothing, \varnothing, s}$-degree
of any Laurent monomial of $R''(k,3)$ equals $1$.
\end{itemize}
\end{lemma}
\begin{proof}
If the block $B$ consists of a single arrow $\arrow{(0,1)}{(1,1)}$, then
equation $F_{\QQ,V,R,B}=1$ is equivalent to $a=a_{1,1}$.
In this case we use the latter equation to exclude the variable~$a$,
and make a change of variables
$$a_{i,j}=a_{i,j}'', \quad i\in [1,k], j\in [1,2], (i,j)\neq (k,2).$$
Put
$$
V''=
\{a_{i,1}''\mid i\in [1,k]\}\cup \{a_{i,2}''\mid i\in [1,k-1]\}.
$$
We define $\psi\colon \TT(V'')\dasharrow \TT(V)$ to be the
change of variables from $a_{i,j}$ to $a''_{i,j}$.
We define the quiver $\QQ''$ so that
$\Ver(\QQ'')=\Ver(\QQ)$
and $\Ar(\QQ'')$ consists of all arrows of $\Ar(\QQ)$
except for the arrow $\arrow{(0,1)}{(1,1)}$. Finally,
we put $R''(i,j)=\psi^*R(i,j)$.
Now the assertion
of the lemma is obvious. Therefore, we assume that the size of the block
$B$ is greater than~$1$, so that $s\ge 2$.

Abusing notation a little bit, we assign
$a_{k,2}=R(k,2)=1$; we do not mean
that $a_{k,2}$ is a variable in this case (in particular, we will
ignore it while computing total degrees with respect to any collection
of variables), but this helps us to keep formulas more neat.
Equation $F_{\QQ,V,R,B}=1$ is equivalent to
\begin{equation}\label{eq:main-variable-horizontal-start}
\frac{a_{1,1}}{a}=
1-
\sum\limits_{\substack{\alpha\in B,\\ \alpha\neq \arrow{(0,1)}{(1,1)}}}
\frac{R\left(h(\alpha)\right)}{R\left(t(\alpha)\right)}=
1-\sum\limits_{i\in [1,s-1],\ j\in [1,2]}
\frac{a_{i+1,j}}{a_{i,j}}.
\end{equation}
We choose $a_{s-1,1}$ to be the weight variable
and~$a$ to be the main variable.

To start with, we make the following change of variables of $V$.
We put $a_{s-1,1}=a_{s-1,1}'$ and we put
\begin{equation}\label{eq:weights-horizontal-start}
a_{i,j}=a_{i,j}'\cdot \left(a_{s-1,1}'\right)^{\wt(i,j)},
\quad
a'=a\cdot \left(a_{s-1,1}'\right)^{\wt(0,1)}
\end{equation}
for the following choice of weights $\wt(i,j)$, $(i,j)\neq (s-1, 1)$.
For any $(i,j)$ with~\mbox{$i\in [1,s]$, $j\in [1,2]$}, and for $(i,j)=(0,1)$
we put
\begin{equation}\label{eq:s-i-horizontal-start}
\wt(i,j)=s-i.
\end{equation}
For any $(i,j)$ with $i\in [s+1, k]$, $j\in [1,2]$, we put $\wt(i,j)=0$.
In particular, this gives~\mbox{$\wt(k,2)=0$}, so that we can
define $a'_{k,2}=a_{k,2}=1$.
Also, \eqref{eq:s-i-horizontal-start} implies that~\mbox{$\wt(s-1,1)=1$},
although we don't mean to use
$\wt(s-1,1)$ in~\eqref{eq:weights-horizontal-start}.
Note that for any arrow $\alpha\in B$ one has
$$\wt\big(t(\alpha)\big)=\wt\big(h(\alpha)\big)+1,$$
and for any
$\alpha\in\In(B)$ one has
$\wt(t(\alpha))=\wt(h(\alpha))$.
In particular, the weight of any non-trivial Laurent monomial appearing
on the right hand side of~\eqref{eq:main-variable-horizontal-start}
equals $-1$.

Put
$$
V'=
\{a_{i,1}'\mid i\in [1,k]\}\cup \{a_{i,2}'\mid i\in [1,k-1]\}\cup\{a'\}.
$$
Define a collection of variables
$W'_{\varnothing,\varnothing,s}$ and a function
$$
\Lambda'_{\varnothing,\varnothing,s}\colon W'_{\varnothing,\varnothing,s}\to\Z
$$
replacing the variables $a_{i,j}$ by $a'_{i,j}$ in
Definition~\ref{definition:00s}.
We rewrite~\eqref{eq:main-variable-horizontal-start} as
\begin{equation}\label{eq:main-variable-1-prelim-horizontal-start-s-ge-3}
\frac{a_{1,1}'}{a'\cdot a_{s-1,1}'}=
1-\frac{1}{a_{s-1,1}'}\cdot\left(
a_{s,1}'+\frac{1}{a_{s-2,1}'}+
\sum\limits_{\substack{i\in [1,s-1],\ j\in [1,2]\\
(i,j)\neq (s-1,1), (s-2,1)
}}
\frac{a_{i+1,j}'}{a_{i,j}'}
\right)
\end{equation}
if $s>2$, and as
\begin{equation}\label{eq:main-variable-1-prelim-horizontal-start-s-2}
\frac{1}{a'\cdot a_{1,1}'}=
1-\frac{1}{a_{1,1}'}\cdot\left(
a_{2,1}'+\frac{a_{2,2}'}{a_{1,2}'}
\right)
\end{equation}
if $s=2$. Note that
the total degree with respect to variables $a_{i,2}'$, $i\in [1,s]$,
of any Laurent monomial appearing on the right hand side
of~\eqref{eq:main-variable-1-prelim-horizontal-start-s-ge-3}
and~\eqref{eq:main-variable-1-prelim-horizontal-start-s-2}
is non-positive; actually, one can make a more precise observation:
the total degree with respect to variables~\mbox{$a_{i,2}'$, $i\in [1,s]$},
of any Laurent monomial appearing on the right hand side
of~\eqref{eq:main-variable-1-prelim-horizontal-start-s-ge-3}
and~\eqref{eq:main-variable-1-prelim-horizontal-start-s-2} is
zero if~\mbox{$s<k$} and is non-positive if $s=k$ (the latter exception appearing
because we ignore~$a_{k,2}=1$ when we compute the total degree).
Similarly, one can check that the~\mbox{$\Lambda'_{\varnothing, \varnothing, s}$-degree}
of any non-trivial
Laurent monomial appearing on the right hand side
of~\eqref{eq:main-variable-1-prelim-horizontal-start-s-ge-3}
and~\eqref{eq:main-variable-1-prelim-horizontal-start-s-2} equals~$1$.

Put $\delta'=a_{1,1}'$ if $s>2$, and put $\delta'=1$ if $s=2$.
By~\eqref{eq:main-variable-1-prelim-horizontal-start-s-ge-3}
and~\eqref{eq:main-variable-1-prelim-horizontal-start-s-2}
we have
\begin{equation}\label{eq:main-variable-1-horizontal-start}
\frac{\delta'}{a'\cdot a_{s-1,1}'}=
1-\frac{1}{a_{s-1,1}'}\cdot\frac{P'}{M'}=
\frac{M'\cdot a_{s-1,1}'-P'}{M'\cdot a_{s-1,1}'}.
\end{equation}
Here $P'$ is a polynomial that depends
only on the
variables $a_{i,j}'$ with $i\in [1,s]$, $j\in [1,2]$,
except for $a_{s-1, 1}'$,
and
$$M'=\prod\limits_{(i,j)\in\mathcal V} a'_{i,j},$$
where
$$\mathcal V=\{(i,j)\mid
i\in [1,s-1], j\in [1,2], (i,j)\neq (s-1,1)\}.$$
As above, the total degree with respect to variables $a_{i,2}'$, $i\in [1,s]$,
of any Laurent monomial of the ratio~\mbox{$P'/M'$},
and thus of any Laurent monomial appearing on the right hand side
of~\eqref{eq:main-variable-1-horizontal-start},
is non-positive.
Similarly, the
$\Lambda'_{\varnothing, \varnothing, s}$-degree of any
Laurent monomial
of the ratio~\mbox{$P'/M'$},
and thus of any Laurent monomial appearing on the right hand side
of~\eqref{eq:main-variable-1-horizontal-start},
equals~$1$.

We rewrite~\eqref{eq:main-variable-1-horizontal-start}
as
\begin{equation}\label{eq:main-variable-2-horizontal-start}
a'=\frac{M'\cdot \delta'}{M'\cdot a_{s-1,1}'-P'}.
\end{equation}
Now we put
\begin{equation}\label{eq:weight-variable-first-step-horizontal-start}
a_{s-1,1}''=\frac{M'\cdot a_{s-1,1}'-P'}{M'},
\end{equation}
and we put $a_{i,j}''=a_{i,j}'$ for all
$i\in [1,k]$, $j\in [1,2]$, such that
$(i,j)\neq (s-1,1)$
(in particular, this gives $a''_{k,2}=a'_{k,2}=1$).
Then
\begin{equation}\label{eq:weight-variable-change-horizontal-start}
a_{s-1,1}=a_{s-1,1}'=\frac{M''\cdot a_{s-1,1}''+P''}{M''},
\end{equation}
where $P''$ and $M''$ are obtained from $P'$ and $M'$ by replacing
the variables $a_{i,j}'$ by the corresponding
variables $a_{i,j}''$,
so that $M''$ is the monomial
$$M''=\prod\limits_{(i,j)\in\mathcal V} a''_{i,j}.$$
Again we observe that the total degree with respect to variables
$a_{i,2}''$, $i\in [1,s]$, of any Laurent monomial of the ratio $P''/M''$,
and thus of any Laurent monomial appearing on the right hand side
of~\eqref{eq:weight-variable-change-horizontal-start},
is non-positive.
Similarly,
we define a collection of variables~\mbox{$W''_{\varnothing,\varnothing,s}$}
and a function
$$
\Lambda''_{\varnothing,\varnothing,s}\colon
W''_{\varnothing,\varnothing,s}\to\Z
$$
replacing the variables $a_{i,j}$ by $a''_{i,j}$ in
Definition~\ref{definition:00s}, and observe that the
$\Lambda''_{\varnothing, \varnothing, s}$-degree of any
Laurent monomial
of the ratio $P''/M''$,
and thus of any Laurent monomial appearing on the right hand side
of~\eqref{eq:weight-variable-change-horizontal-start},
equals~$1$.

We can rewrite~\eqref{eq:main-variable-2-horizontal-start} as
\begin{equation}\label{eq:main-variable-3-horizontal-start}
a'=\frac{\delta''}{a_{s-1,1}''},
\end{equation}
where $\delta''=a_{1,1}''$ if $s>2$, and $\delta''=1$ if $s=2$.
We see that
$$\deg_{\Lambda''_{\varnothing, \varnothing, s}}(\delta'')=2-s.$$

By~\eqref{eq:weights-horizontal-start}
and~\eqref{eq:weight-variable-change-horizontal-start}
one has
\begin{equation}\label{eq:i-1-horizontal-start}
a_{i,1}=a_{i,1}'\cdot \left(a_{s-1,1}'\right)^{s-i}=
a_{i,1}''\cdot\left(\frac{M''\cdot a_{s-1,1}''+P''}{M''}\right)^{s-i}
\end{equation}
for any $i\in [1,s]$, $i\neq s-1$.
Also,~\eqref{eq:weights-horizontal-start} implies that
\begin{equation}\label{eq:i-2-horizontal-start}
a_{i,2}=a_{i,2}'\cdot \left(a_{s-1,1}'\right)^{s-i}=
a_{i,2}''\cdot\left(\frac{M''\cdot a_{s-1,1}''+P''}{M''}\right)^{s-i}
\end{equation}
for any $ i\in [1,s]$.
Finally, \eqref{eq:weights-horizontal-start}
and~\eqref{eq:main-variable-3-horizontal-start}
imply that
\begin{equation}\label{eq:a-horizontal-start}
a=a'\cdot \left(a_{s-1,1}'\right)^{s}=
\frac{\delta''}{a_{s-1,1}''}\cdot\left(\frac{M''\cdot a_{s-1,1}''+P''}{M''}\right)^{s}.
\end{equation}
Once again we notice that the total degree with respect to variables
$a_{i,2}''$, $i\in [1,s]$, of any Laurent monomial
appearing on the right hand side of~\eqref{eq:a-horizontal-start}
is non-positive.
Also, we see that the right hand side
of~\eqref{eq:a-horizontal-start} does not depend on variables
$a''_{i,j}$ with~\mbox{$i\in [s+1, k]$, $j\in [1,2]$}, and each of its Laurent monomials
has non-negative degree in each of the variables~\mbox{$a''_{s,j}$, $j\in [1,2]$}.
Similarly, we compute
$$
\deg_{\Lambda''_{\varnothing, \varnothing, s}}
\left(\frac{\delta''}{a''_{s-1,1}}\right)=1-s,
$$
and see that the
$\Lambda''_{\varnothing, \varnothing, s}$-degree of any
Laurent monomial appearing on the right hand side
of~\eqref{eq:a-horizontal-start} equals~$1$.

Equation~\eqref{eq:a-horizontal-start}
allows us to exclude the variable
$a$. Now we are going to show that after making this exclusion
and changing variables $a_{i,j}$ to 
$a_{i,j}''$ the Laurent polynomial~$F_{\QQ, V, R}$
remains a Laurent polynomial.

Let $\alpha=\arrow{(i,1)}{(i,2)}$ be an inner arrow for $B$.
Suppose that $i\neq s-1$.
Then
\begin{equation}\label{eq:inner-1-horizontal-start}
\frac{R\left(h(\alpha)\right)}{R\left(t(\alpha)\right)}=
\frac{a_{i,2}}{a_{i,1}}=\frac{a'_{i,2}}{a'_{i,1}}=
\frac{a''_{i,2}}{a''_{i,1}}.
\end{equation}
If $\alpha=\arrow{(s-1,1)}{(s-1,2)}$, then
\begin{equation}\label{eq:inner-3-horizontal-start}
\frac{R\left(h(\alpha)\right)}{R\left(t(\alpha)\right)}=
\frac{a_{s-1,2}}{a_{s-1,1}}=
\frac{a'_{s-1,2}\cdot a'_{s-1,1}}{a'_{s-1,1}}=a'_{s-1,2}=a''_{s-1,2}.
\end{equation}

Put
$$
V''=
\{a_{i,1}''\mid i\in [1,k]\}\cup \{a_{i,2}''\mid i\in [1,k-1]\}.
$$
We define
$$\psi\colon \TT(V'')\dasharrow \TT(V)$$
to be the change of variables from $a_{i,j}$ to $a''_{i,j}$.
We define the quiver $\QQ''$ so that~\mbox{$\Ver(\QQ'')=\Ver(\QQ)$} and
$\Ar(\QQ'')=\Ar(\QQ)\setminus B$. Finally, we put $R''(i,j)=\psi^*R(i,j)$.

Denote by $\alpha_{f}$ the arrow $\arrow{(k,2)}{(k,3)}$.
Denote by $C$ the set of
arrows~\mbox{$\alpha\in\Ar(\QQ)\setminus\{\alpha_f\}$}
such that $h(\alpha)=(i,j)$
for some $i\in [s+1, k]$, $j\in [1,2]$.
Then the set $\Ar(\QQ)$ is a disjoint union
of the sets $B$, $\In(B)$, $C$ and $\{\alpha_f\}$.

One has
\begin{multline*}
\psi^*F_{\QQ,V,R}=\psi^*\left(\sum\limits_{\alpha\in\Ar(\QQ)}
\frac{R\left(h(\alpha)\right)}{R\left(t(\alpha)\right)}\right)=\\
=
\psi^*\left(
\sum\limits_{\alpha\in B}
\frac{R\left(h(\alpha)\right)}{R\left(t(\alpha)\right)}
+
\sum\limits_{\alpha\in\In(B)}
\frac{R\left(h(\alpha)\right)}{R\left(t(\alpha)\right)}
+
\sum\limits_{\alpha\in C}
\frac{R\left(h(\alpha)\right)}{R\left(t(\alpha)\right)}
+
\frac{R\left(h(\alpha_f)\right)}{R\left(t(\alpha_f)\right)}
\right)
=\\ =
1+
\psi^*\left(
\sum\limits_{\alpha\in\In(B)}
\frac{R\left(h(\alpha)\right)}{R\left(t(\alpha)\right)}
+
\sum\limits_{\alpha\in C}
\frac{R\left(h(\alpha)\right)}{R\left(t(\alpha)\right)}
+
\frac{R\left(h(\alpha_f)\right)}{R\left(t(\alpha_f)\right)}
\right).
\end{multline*}

If $\alpha\in\In(B)$,
then
$\psi^*\left(
\frac{R\left(h(\alpha)\right)}{R\left(t(\alpha)\right)}\right)$
is a Laurent monomial in the variables of $V''$
by~\eqref{eq:inner-1-horizontal-start}
and~\eqref{eq:inner-3-horizontal-start}.

If $\alpha\in C$, then
$$
\psi^*\left(\frac{R\left(h(\alpha)\right)}{R\left(t(\alpha)\right)}\right)=
\psi^*\left(\frac{a_{h(\alpha)}}{a_{t(\alpha)}}\right)=
\frac{a''_{h(\alpha)}}{a''_{t(\alpha)}}
$$
by conditions~(ii) and~(iii), because
the variables $a_{i,j}$ with $i\in [s,k]$, $j\in [1,2]$,
were not changed when passing from
$V$ to $V'$ and further to $V''$.

Finally one can notice that
$$
\psi^*\left(\frac{R\left(h(\alpha_f)\right)}
{R\left(t(\alpha_f)\right)}\right)=
\psi^*R(k,3)=\psi^*a
$$
is a Laurent polynomial in the variables of $V''$
by~\eqref{eq:a-horizontal-start}.

Therefore, we see that
$\psi^*F_{\QQ,V,R}$ is a Laurent polynomial in the variables of $V''$.

Note that assertion (I) of the lemma
holds by definition of $\QQ''$.
The variables $a_{i,j}\in V$ with $i\in [s, k]$, $j\in [1,2]$,
were not changed when passing from
$V$ to $V'$ and further to~$V''$.
From this we conclude that
assertions~(II) and~(III) of the lemma hold.
Assertions~\mbox{(IV),~(V)} and~(VI) hold
due to equations~\eqref{eq:i-1-horizontal-start},
\eqref{eq:weight-variable-change-horizontal-start}
and~\eqref{eq:i-2-horizontal-start},
respectively.
Finally, validity of assertions~(VII), (VIII) and~(IX)
follows from~\eqref{eq:a-horizontal-start}.
\end{proof}

\begin{remark}
\label{remark:horizontal-block-G2-start}
In Lemma~\ref{lemma:horizontal-block-G2-start} we worked with a hypersurface given by equation
$$1-F_{\QQ,V,R,B}=0.$$ However, in the proof of Proposition~\ref{proposition: periods} we will need
to analyze a tubular neighborhood of the latter hypersurface, i.\,e. to work with an equation
$1-F_{\QQ,V,R,B}=U$. In this situation equation~\eqref{eq:main-variable-2-horizontal-start}
takes the form
\begin{equation}\label{eq:main-variable-2-horizontal-start-U}
a'=\frac{M'\cdot \delta'}{(1-U)\cdot M'\cdot a_{s-1,1}'-P'}.
\end{equation}
\end{remark}

In the rest of this section, as well as in Sections~\ref{section:mixed}
and~\ref{section:vertical}, we will have to keep track of
the history of changes of variables performed up to some point.
This will be done with the help of the following definitions.

\begin{definition}\label{definition:history}
Let $n=2$. Let $\QQ$ be an admissible quiver,
and let $r\in [1,k]$.
We say that~\mbox{$(\MM,\WW, \gamma)$}, where
$\gamma\in [1,r]$ and $\MM$ and $\WW$ are subsets of $[1,\gamma-1]$, is a
\emph{block history} of~\mbox{$(\QQ,r)$} if the following conditions hold:
\begin{itemize}
\item
one has $\WW=\varnothing$ if and only if $\gamma=1$ (so that
one also has $\MM=\varnothing$ in this case);
\item
if $\WW\neq\varnothing$, then $|\MM|=|\WW|-1$, and
one has $\MM=\{m_1,\ldots, m_{|\MM|}\}$
and~\mbox{$\WW=\{w_0,w_1,\ldots,w_{|\MM|}\}$} with
$$w_0<w_0+1=m_1< w_1<\ldots < w_{|\MM|-1}+1=
m_{|\MM|}< w_{|\MM|}<w_{|\MM|}+1=\gamma.$$
\end{itemize}
\end{definition}

\begin{definition}\label{definition:MWgamma}
Let $r\in [1,k]$, and let $(\MM,\WW, \gamma)$ be a
block history of $(\QQ,r)$.
Put
$$W_{\MM, \WW, \gamma,r}=\{a_{i,1}\mid i\in [1,r]\}\cup
\{a_{i,2}\mid i\in [1,\gamma-1]\setminus\MM\}\cup\{a_{r,2}\}$$
if $r<k$, and put
$$W_{\MM, \WW, \gamma,k}=\{a_{i,1}\mid i\in [1,k]\}\cup
\{a_{i,2}\mid i\in [1,\gamma-1]\setminus\MM\}\}$$
if $r=k$.
For any $i<\max\WW$ (i.\,e. for any $i<\gamma-1$)
we put $w(i)=\min\{w\in\WW\mid w>i\}$.
We define a function
$$
\Lambda_{\MM,\WW,\gamma,r}\colon
W_{\MM,\WW,\gamma,r}\to\Z
$$
as follows.
If $i\in\WW$, then
we put $\Lambda_{\MM,\WW,\gamma,r}(a_{i,1})=1$ and
$\Lambda_{\MM,\WW,\gamma,r}(a_{i,2})=-1$.
If~\mbox{$i\not\in\WW$} and $i<\gamma-1$,
then we put
$$\Lambda_{\MM,\WW,\gamma,r}(a_{i,1})=i-w(i)$$
and
$$\Lambda_{\MM,\WW,\gamma,r}(a_{i,2})=i-w(i)-1.$$
If~\mbox{$\gamma\le i\le r$}, then we put
$\Lambda_{\MM,\WW,\gamma,r}(a_{i,1})=1$.
Finally, if $r<k$, then
we put~\mbox{$\Lambda_{\MM,\WW,\gamma,r}(a_{r,2})=0$}.

If the variables of the set $W_{\MM,\WW,\gamma,r}$ are clearly labeled by some
set of indices $\{(i,j)\}$ we will sometimes write $\Lambda_{\MM,\WW,\gamma,r}(i,j)$
instead of $\Lambda_{\MM,\WW,\gamma,r}(a_{i,j})$.
\end{definition}

Note that Definition~\ref{definition:00s} is a particular case of
Definition~\ref{definition:MWgamma} for $\MM=\WW=\varnothing$
and~\mbox{$\gamma=r=s$}.

The following elementary observation will be rather useful
for the remaining lemmas of this section.

\begin{remark}\label{remark:stupid}
Let $V$ be a set of variables, and $F$ be a Laurent polynomial in the variables
of $V$. Let $V'$ be some other set of variables. Consider a rational map
$\psi\colon \TT(V')\dasharrow \TT(V)$. Let $W\subset V$ and $W'\subset V'$ be some
subsets of variables. Choose two functions
$\Lambda\colon W\to\Z$ and $\Lambda'\colon W'\to\Z$.
Suppose that the rational
function $\psi^*F$ is a Laurent
polynomial in the variables of $V'$.
Suppose that for any $a\in W$ the rational function
$\psi^*a$ is a Laurent
polynomial in the variables of $V'$, and for any Laurent monomial
$\mu'$ of $\psi^*a$ one has
$$\deg_{\Lambda'}(\mu')=\deg_{\Lambda}(a).$$
Then for any Laurent monomial $\mu$ of $F$ one has
$$\deg_{\Lambda'}(\psi^*\mu)=\deg_{\Lambda}(\mu).$$
\end{remark}

Now we return to our changes of variables.

\begin{lemma}
\label{lemma:horizontal-block-G2}
Suppose that $n=2$.
Let $(\QQ, V, R)$ be a triplet, and $B\subset\Ar(\QQ)$ be a horizontal
block such that the arrow $\arrow{(0,1)}{(1,1)}$ is not contained in $B$.
Let $r$ be the smallest number such that
$(r,1)\in\Ver(B)$. Suppose that the size of the block $B$
is greater than $1$, so that $B$ is not a basic block.

Suppose that there is a
block history $(\MM, \WW, r)$, i.\,e. one with $\gamma=r$
in the notation of Definition~\ref{definition:history},
of $(\QQ,r)$ such that $V$ is a set of variables
$$V=\{a_{i,1}\mid i\in [1,k]\}\cup \{a_{i,2}\mid i\in [1,k-1]\setminus\MM\}.$$
Suppose that there are rational functions
$\bar{R}(i)$, $i\in [1,r-1]$, in the variables of $V$
such that the following conditions hold:
\begin{itemize}
\item[(i)]
the quiver $\QQ$ does not
contain vertical arrows $\alpha$ such that $h(\alpha)=(i,j)$ for
$i\in [1,r]$, $j\in [1,2]$;
\item[(ii)]
one has $R(k,2)=1$ (in what follows, we assign $a_{k,2}=R(k,2)=1$,
abusing notation a little bit);
\item[(iii)]
for $(i,j)$ with $i\in [r, k]$,
$j\in [1,2]$, $(i,j)\neq (k,2)$,
one has $R(i,j)=a_{i,j}$;
\item[(iv)] for any $i\in [1,r-1]\setminus\WW$ one has
$R(i,1)=a_{i,1}\cdot\bar{R}(i)$;
\item[(v)] for any $i\in \WW$ one has
$R(i,1)=\bar{R}(i)$;
\item[(vi)] for any $i\in [1,r-1]\setminus\MM$ one has
$R(i,2)=a_{i,2}\cdot\bar{R}(i)$;
\item[(vii)] for any $i\in \MM$ one has
$$R(i,2)=\frac{a_{i+1,2}}{a_{w(i),1}}
\cdot\bar{R}(i),$$
where
$w(i)=\min\{w\in\WW\mid w>i\}$;
\item[(viii)] the rational function $R(k,3)$ is a Laurent
polynomial in the variables of $V$ such that
$R(k,3)$ does not depend on variables $a_{i,j}$ with $i\in [r+1, k]$,
$j\in [1,2]$, and each of its
Laurent monomials has non-negative degree in each
of the variables $a_{r,j}$, $j\in [1,2]$;
\item[(ix)] the total degree of any Laurent monomial of $R(k,3)$
with respect to variables $a_{i,2}$, $i\in [1,r]\setminus\MM$, is non-positive;
\item[(x)] the $\Lambda_{\MM,\WW,r,r}$-degree of any Laurent
monomial of $R(k,3)$ equals~$1$.
\end{itemize}

Then there exists a
change of variables $\psi$ that agrees with the triplet $(\QQ,V,R)$
and with the
block $B$ with the following properties.
Let $(\QQ'',V'',R'')$ be the transformation of the triplet
$(\QQ,V,R)$ associated to $\psi$,
and let $s$ be the largest number
such that $(s,1)\in\Ver(B)$.
Then $\psi^*F_{\QQ,V,R}$ is a Laurent polynomial
in the variables of $V''$.
Moreover, there is a block history $(\MM'', \WW'',s)$
of $(\QQ'', s)$ with $\MM''=\MM\cup\{r\}$ such that
$V''$ is a set of variables
$$V''=\{a''_{i,1}\mid i\in [1,k]\}\cup
\{a''_{i,2}\mid i\in [1,k-1]\setminus\MM''\},$$
and conditions (i)--(x) hold after replacing
$\QQ$, $V$, $R$, $\MM$, $\WW$ and $r$
by~$\QQ''$, $V''$, $R''$, $\MM''$, $\WW''$ and~$s$,
respectively.
\end{lemma}
\begin{proof}
The proof is similar to that of Lemma~\ref{lemma:horizontal-block-G2-start},
with the only difference that we choose $a_{r,2}$ to be the main variable
and $a_{s-1,1}$ to be the weight variable.
\end{proof}


\begin{lemma}
\label{lemma:horizontal-block-G2-narrow}
Suppose that $n=2$.
Let $(\QQ, V, R)$ be a triplet, and $B\subset\Ar(\QQ)$ be a basic horizontal
block such that the arrow $\arrow{(0,1)}{(1,1)}$ is not contained in $B$.
Let $r$ be the smallest number such that
$(r,1)\in\Ver(B)$, so that $B$ is the $r$-th basic horizontal
block with $r\ge 1$.

Suppose that there is a
block history $(\MM, \WW, \gamma)$
of $(\QQ,r)$ such that $V$ is a set of variables
$$V=\{a_{i,1}\mid i\in [1,k]\}\cup
\{a_{i,2}\mid i\in [1,\gamma-1]\setminus\MM\}\cup
\{a_{i,2}\mid i\in [r,k-1]\}.$$
Suppose that there are rational functions
$\bar{R}(i)$, $i\in [1,\gamma-1]$, in the variables
of $V$ such that the following conditions hold:
\begin{itemize}
\item[(i)]
the quiver $\QQ$ does not
contain vertical arrows $\alpha$ such that $h(\alpha)=(i,j)$ for
$i\in [1,r]$, $j\in [1,2]$;
\item[(ii)]
one has $R(k,2)=1$ (in what follows, we assign $a_{k,2}=R(k,2)=1$,
abusing notation a little bit);
\item[(iii)]
for $(i,j)$ with $i\in [r, k]$,
$j\in [1,2]$, $(i,j)\neq (k,2)$,
one has $R(i,j)=a_{i,j}$;
\item[(iv)] for any $i\in [1,\gamma-1]\setminus\WW$ one has
$R(i,1)=a_{i,1}\cdot\bar{R}(i)$;
\item[(v)] for any $i\in \WW$ one has
$R(i,1)=\bar{R}(i)$;
\item[(vi)] for any $i\in [1,\gamma-1]\setminus\MM$ one has
$R(i,2)=a_{i,2}\cdot\bar{R}(i)$;
\item[(vii)] for any $i\in \MM$ one has
$$R(i,2)=\frac{a_{i+1,2}}{a_{w(i),1}}
\cdot\bar{R}(i),$$
where
$w(i)=\min\{w\in\WW\mid w>i\}$;
\item[(viii)] for any $i\in [\gamma, r-1]$ one has
$$R(i,1)=a_{r,1}+a_{r-1,1}+\ldots+a_{i,1};$$
\item[(ix)] for any $i\in [\gamma, r-1]$ one has
$$R(i,2)=\frac{a_{r,2}\cdot
(a_{r,1}+a_{r-1,1})\cdot
(a_{r,1}+a_{r-1,1}+a_{r-2,1})\cdot\ldots\cdot
(a_{r,1}+a_{r-1,1}+\ldots+a_{i,1})}
{a_{r-1,1}\cdot a_{r-2,1}\cdot\ldots\cdot a_{i,1}};$$
\item[(x)] the rational function $R(k,3)$ is a Laurent
polynomial in the variables of $V$ such that
$R(k,3)$ does not depend on variables $a_{i,j}$ with $i\in [r+1, k]$,
$j\in [1,2]$, and each of its
Laurent monomials has non-negative degree in each
of the variables $a_{r,j}$, $j\in [1,2]$;
\item[(xi)] the total degree of any Laurent monomial of $R(k,3)$
with respect to the variables $a_{i,2}$, $i\in [1,\gamma-1]\setminus\MM$, is non-positive;
\item[(xii)] the $\Lambda_{\MM,\WW,\gamma,r}$-degree
of any Laurent monomial of $R(k,3)$ equals~$1$.
\end{itemize}

Then there exists a
change of variables $\psi$ that agrees with the triplet $(\QQ,V,R)$
and with the
block $B$ with the following properties.
Let $(\QQ',V',R')$ be the transformation of the triplet
$(\QQ,V,R)$ associated to $\psi$.
Then $\psi^*F_{\QQ,V,R}$ is a Laurent polynomial
in the variables of $V'$.
Moreover, $V'$ is a set of variables
$$V'=\{a_{i,1}'\mid i\in [1,k]\}\cup
\{a_{i,2}'\mid i\in [1,\gamma-1]\setminus\MM\}\cup
\{a_{i,2}'\mid i\in [r+1,k-1]\},$$
and conditions (i)--(xii) hold after replacing
$\QQ$, $V$, $R$, and $r$ by $\QQ'$, $V'$, $R'$,
and $r+1$, respectively, and keeping $\MM$, $\WW$ and $\gamma$ the same
as before.
\end{lemma}
\begin{proof}
The proof is similar to the proof of Lemma~\ref{lemma:horizontal-block-G2-start}.
Here we choose $a_{r,2}$ to be the main variable,
and do not need a weight variable at all.
\end{proof}


\section{Mixed blocks}
\label{section:mixed}

In this section we deal with changes of variables that agree with
mixed blocks for Grassmannians $\G(2,k+2)$.

\begin{lemma}
\label{lemma:mixed-block-G2-start}
Suppose that $n=2$.
Let $(\QQ, V, R)$ be a triplet, and $B\subset\Ar(\QQ)$ be a mixed
block such that the arrow $\arrow{(0,1)}{(1,1)}$ is contained in $B$
and the arrow $\arrow{(k,2)}{(k,3)}$ is not contained in $B$.
Suppose that $V$ is a set of variables
$$V=\{a_{i,1}\mid i\in [1,k]\}\cup \{a_{i,2}\mid i\in [1,k-1]\}\cup\{a\},$$
and the following conditions hold:
\begin{itemize}
\item[(i)]
one has $R(k,2)=1$;
\item[(ii)]
$R(0,1)=R(k,3)=a$;
\item[(iii)]
for $i\in [1,k]$, $j\in [1,2]$, $(i,j)\neq (k,2)$ one has
$R(i,j)=a_{i,j}$, see Figure~\ref{figure:g25one}
\end{itemize}

Then there exists a
change of variables $\psi$ that agrees with the triplet $(\QQ,V,R)$
and with the
block $B$ such that for the transformation
$(\QQ'',V'',R'')$ of the triplet
$(\QQ,V,R)$ associated to $\psi$
the rational function $\psi^*F_{\QQ,V,R}$ is a Laurent polynomial
in the variables of $V''$.
\end{lemma}
\begin{proof}
The proof is similar to the proof of Lemma~\ref{lemma:horizontal-block-G2-start}.
Here we choose $a$ to be the main variable
and $a_{k-1,2}$ to be the weight variable.
\end{proof}


\begin{lemma}
\label{lemma:mixed-block-G2}
Suppose that $n=2$.
Let $(\QQ, V, R)$ be a triplet, and $B\subset\Ar(\QQ)$ be a mixed
block such that the arrow $\arrow{(0,1)}{(1,1)}$ is not contained in $B$.
Let $r$ be the smallest number such that there
is a vertical arrow $\arrow{(r,1)}{(r+1,1)}\in B$.

Suppose that there is a
block history
$(\MM, \WW, r)$, i.\,e. one with $\gamma=r$
in the notation of Definition~\ref{definition:history},
of $(\QQ,r)$ such that $V$ is a set of variables
$$V=\{a_{i,1}\mid i\in [1,k]\}\cup \{a_{i,2}\mid i\in [1,k-1]\setminus\MM\}.$$
Suppose that there are rational functions
$\bar{R}(i)$, $i\in [1,r-1]$, in the variables of~$V$
such that the following conditions hold:
\begin{itemize}
\item[(i)]
the quiver $\QQ$ does not
contain vertical arrows $\alpha$ such that $h(\alpha)=(i,j)$ for
$i\in [1,r]$, $j\in [1,2]$;
\item[(ii)]
one has $R(k,2)=1$;
\item[(iii)]
for $(i,j)$ with $i\in [r, k]$,
$j\in [1,2]$, $(i,j)\neq (k,2)$,
one has $R(i,j)=a_{i,j}$;
\item[(iv)] for any $i\in [1,r-1]\setminus\WW$ one has
$R(i,1)=a_{i,1}\cdot\bar{R}(i)$;
\item[(v)] for any $i\in \WW$ one has
$R(i,1)=\bar{R}(i)$;
\item[(vi)] for any $i\in [1,r-1]\setminus\MM$ one has
$R(i,2)=a_{i,2}\cdot\bar{R}(i)$;
\item[(vii)] for any $i\in \MM$ one has
$$R(i,2)=\frac{a_{i+1,2}}{a_{w(i),1}}
\cdot\bar{R}(i),$$
where
$w(i)=\min\{w\in\WW\mid w>i\}$;
\item[(viii)] the rational function $R(k,3)$ is a Laurent
polynomial in the variables of $V$ such that
$R(k,3)$ does not depend on variables $a_{i,j}$ with $i\in [r+1, k]$,
$j\in [1,2]$, and each of its Laurent
monomials has non-negative degree in each
of the variables $a_{r,j}$, $j\in [1,2]$;
\item[(ix)] the total degree of any Laurent monomial of $R(k,3)$
with respect to variables $a_{i,2}$, $i\in [1,r]\setminus\MM$, is non-positive.
\end{itemize}

Then there exists a
change of variables $\psi$
that agrees with the triplet $(\QQ,V,R)$ and with the
block $B$ such that for the transformation
$(\QQ'',V'',R'')$ of the triplet
$(\QQ,V,R)$ associated to $\psi$
the rational function $\psi^*F_{\QQ,V,R}$ is a Laurent polynomial
in the variables of $V''$.
\end{lemma}
\begin{proof}
The proof is similar to the proof of Lemma~\ref{lemma:horizontal-block-G2-start}.
Here we choose $a_{r,1}$ to be the main variable
and $a_{k,1}$ to be the weight variable.
\end{proof}


\section{Vertical blocks}
\label{section:vertical}

In this section we deal with changes of variables that agree with
vertical blocks for Grassmannians $\G(2,k+2)$ and make some concluding remarks
on the changes of variables that agree with various kinds
of blocks.

\begin{lemma}
\label{lemma:vertical-block-G2}
Suppose that $n=2$.
Let $(\QQ, V, R)$ be a triplet, and $B\subset\Ar(\QQ)$ be a vertical
block such that the arrow $\arrow{(k,2)}{(k,3)}$ is
not contained in $B$ (i.\,e. $B$ is the first basic vertical block).

Suppose that there is a
block history $(\MM, \WW, \gamma)$
of $(\QQ,k)$ such that $V$ is a set of variables
$$V=\{a_{i,1}\mid i\in [1,k]\}\cup
\{a_{i,2}\mid i\in [1,\gamma-1]\setminus\MM\}.$$
Suppose that there are rational functions
$\bar{R}(i)$, $i\in [1,\gamma-1]$, in the variables
of~$V$ such that the following conditions hold:
\begin{itemize}
\item[(i)]
the quiver $\QQ$ does not
contain vertical arrows;
\item[(ii)]
one has $R(k,2)=1$;
\item[(iii)]
one has $R(k,1)=a_{k,1}$;
\item[(iv)] for any $i\in [1,\gamma-1]\setminus\WW$ one has
$R(i,1)=a_{i,1}\cdot\bar{R}(i)$;
\item[(v)] for any $i\in \WW$ one has
$R(i,1)=\bar{R}(i)$;
\item[(vi)] for any $i\in [1,\gamma-1]\setminus\MM$ one has
$R(i,2)=a_{i,2}\cdot\bar{R}(i)$;
\item[(vii)] for any $i\in \MM$ one has
$$R(i,2)=\frac{a_{i+1,2}}{a_{w(i),1}}
\cdot\bar{R}(i),$$
where
$w(i)=\min\{w\in\WW\mid w>i\}$;
\item[(viii)] for any $i\in [\gamma, k-1]$ one has
$$R(i,1)=a_{k,1}+a_{k-1,1}+\ldots+a_{i,1};$$
\item[(ix)] for any $i\in [\gamma, k-1]$ one has
$$R(i,2)=\frac{
(a_{k,1}+a_{k-1,1})\cdot
(a_{k,1}+a_{k-1,1}+a_{k-2,1})\cdot\ldots\cdot
(a_{k,1}+a_{k-1,1}+\ldots+a_{i,1})}
{a_{k-1,1}\cdot a_{k-2,1}\cdot\ldots\cdot a_{i,1}};$$
\item[(x)] the rational function $R(k,3)$ is a Laurent
polynomial in the variables of $V$ such that
the $\Lambda_{\MM,\WW,\gamma,k}$-degree of any Laurent monomial of $R(k,3)$
is non-negative.
\end{itemize}

Then there exists a
change of variables $\psi$ that agrees with the triplet $(\QQ,V,R)$
and with the
block $B$ such that for the transformation
$(\QQ'',V'',R'')$ of the triplet
$(\QQ,V,R)$ associated to $\psi$
the rational function $\psi^*F_{\QQ,V,R}$ is a Laurent polynomial
in the variables of $V''$.
\end{lemma}
\begin{proof}
The proof is similar to the proof of Lemma~\ref{lemma:horizontal-block-G2-start}.
Here we define the main and the weight variables as follows.
If $\gamma<k$, we put $u=k$. If $\gamma=k$,
then $\WW\neq\varnothing$, and we put
$u=\min\WW$ (so that $u\neq\gamma$ by definition of block history).
In both cases we choose $a_{u,1}$ to be the weight variable
and~$a_{\gamma,1}$ to be the main variable.
The weights (cf.~\eqref{eq:s-i-horizontal-start}) are defined as
\begin{equation}\label{eq:wt-vs-Lambda}
\wt(i,j)=\Lambda_{\MM,\WW,\gamma,k}(i,j)
\end{equation}
for
$$
(i,j)\in\{(i,1)\mid i\in [1,k]\}\cup\{(i,2)\mid i\in [1,\gamma-1]\setminus\MM\}.
$$
\end{proof}


We conclude this section by a couple of general remarks concerning the proofs
of Lemmas~\ref{lemma:horizontal-block-G2-start},
\ref{lemma:horizontal-block-G2},
\ref{lemma:horizontal-block-G2-narrow},
\ref{lemma:mixed-block-G2-start},
\ref{lemma:mixed-block-G2}
and~\ref{lemma:vertical-block-G2}.

\begin{remark}\label{remark:lemmas}
Let $n=2$, let $(\QQ,V,R)$ be a triplet, and
let $B\subset\Ar(\QQ)$ be a block such that a quiver with the set
of vertices coinciding with $\Ver(\QQ)$ and the set
of arrows $\Ar(\QQ)\setminus B$ is admissible.

Suppose that $(\QQ,V,R)$ is a result of an application
of Lemma~\ref{lemma:horizontal-block-G2-start}
to a triplet
$(\widehat{\QQ},\widehat{V},\widehat{R})$
and some horizontal block $\widehat{B}\subset\Ar(\widehat{\QQ})$
with the properties described in the assumptions of
Lemma~\ref{lemma:horizontal-block-G2-start}.
Then $(\QQ,V,R)$ and $B$
satisfy the assumptions of Lemmas~\ref{lemma:horizontal-block-G2},
\ref{lemma:horizontal-block-G2-narrow},
\ref{lemma:mixed-block-G2}
or~\ref{lemma:vertical-block-G2}
depending on whether $B$ is a horizontal block of size at least~$2$,
a basic horizontal block, a mixed block or a vertical block.

Similarly,
suppose that $(\QQ,V,R)$ is a result of an application
of Lemma~\ref{lemma:horizontal-block-G2} to some triplet
$(\widehat{\QQ},\widehat{V},\widehat{R})$ and some horizontal block
$\widehat{B}\subset\Ar(\widehat{\QQ})$.
Then $(\QQ,V,R)$ and $B$
satisfy the assumptions of
Lemmas~\ref{lemma:horizontal-block-G2},
\ref{lemma:horizontal-block-G2-narrow},
\ref{lemma:mixed-block-G2}
or~\ref{lemma:vertical-block-G2}
depending on whether $B$ is a horizontal block of size at least~$2$,
a basic horizontal block, a mixed block or a vertical block.

Finally, suppose that $(\QQ,V,R)$ is a result of an application
of Lemma~\ref{lemma:horizontal-block-G2-narrow}.
Then~\mbox{$(\QQ,V,R)$} and $B$
satisfy the assumptions of
Lemmas~\ref{lemma:horizontal-block-G2-narrow},
\ref{lemma:mixed-block-G2}
or~\ref{lemma:vertical-block-G2}
depending on whether~$B$ is
a basic horizontal block, a mixed block or a vertical block.
Note that we are not able to apply Lemma~\ref{lemma:horizontal-block-G2}
to a result of an application of Lemma~\ref{lemma:horizontal-block-G2-narrow}.

We will use these observations during
the inductive proof of Theorem~\ref{theorem: main} below.
\end{remark}

\begin{remark}
A reader may have an impression that our choice
of main variables and weight variables in the proofs
of Lemmas~\ref{lemma:horizontal-block-G2-start},
\ref{lemma:horizontal-block-G2},
\ref{lemma:mixed-block-G2-start},
\ref{lemma:mixed-block-G2}
and~\ref{lemma:vertical-block-G2} is rather arbitrary. This is true to some
extent, and some choices could be maid in some other way.
Nevertheless, at least part of our choices is inevitable, and
some of the others are done due to our attempts to optimize
the computations. First,
when we choose a main variable for some block we want that
the corresponding vertex is not simultaneously
a tail of some arrow of the block and
a head of some other arrow of the block.
Thus one of the very few unnecessary things here is the choice of the variable
$a$ instead of $a_{1,2}$ as a main variable in the proof of
Lemma~\ref{lemma:horizontal-block-G2-start}.
We did this because we wanted to unify the case when the size
of the block equals~$1$ and the case when the size
of the block exceeds~$1$, and also, more
importantly, to obtain a bit more uniform set of variables
after this first change.

Furthermore, the weight of a weight variable
with respect to itself is~$1$,
and our method of expressing a main variable requires that
for any arrow $\alpha$ of the block one has $\wt(t(\alpha))=\wt(h(\alpha))+1$.
Therefore, when working with a horizontal block
$B$ in the proof of Lemma~\ref{lemma:horizontal-block-G2}
we choose a weight variable in the second line from below in~\mbox{$\Ver(B)$};
this allows us to leave the variables corresponding to the last
row of $\Ver(B)$ unaffected
by the change of coordinates, so that the further
changes of coordinates remain relatively simple,
and so that we do not have a contradiction
with assigning the weight to $a_{k,2}=1$ if~\mbox{$(k,2)\in\Ver(B)$}.
Also, in this case we choose a weight variable in the first column rather
than in the second column of $\Ver(B)$ to avoid
dealing with more cases that would arise if
the block $B$ could contain an arrow between the vertices
corresponding to a main variable and a weight variable.

Similarly to this,
when we choose a weight variable in the proof of
Lemma~\ref{lemma:mixed-block-G2-start}, we choose it in the second column
to avoid dealing with more cases that would arise if
there could be an arrow in the block $B$ between the vertices
corresponding to the main variable $a$ and a weight variable.
Besides this, we are forced to choose the weight variable in the $(k-1)$-th
row since the weight of $a_{k,2}=1$ with respect to anything is $0$.

In the proof of Lemma~\ref{lemma:mixed-block-G2}
our main variable corresponds, as explained above,
to the unique vertex $(i,j)\in\Ver(B)$ that is not simultaneously
a tail of some arrow of the block and
a head of some other arrow of $B$, and such that the corresponding
rational function~\mbox{$R(i,j)$} is a variable.
On the other hand, the choice of the weight variable
is dictated by the requirement that the distance between the
corresponding vertex and the vertex~\mbox{$(k,2)\in\Ver(B)$}
along the arrows
of $B$ should equal $1$. This leaves us with
a choice between the variables $a_{k,1}$ and $a_{k-1,2}$,
with no big difference between these cases.

Finally, in the proof of Lemma~\ref{lemma:vertical-block-G2}
we choose $a_{\gamma,1}$ to be the
main variable since it is the only one that we actually managed to express
via the remaining variables in the most general case.
The choice of the weight variable $a_{u,1}$ is mostly defined
by the function~\mbox{$\Lambda_{\MM,\WW,\gamma,k}$}.
In principle, $u$ can be replaced
by any number from the set $\WW$, or by any number
from the set $[\gamma+1,k]$.
\end{remark}

\begin{remark}
Our proofs of Lemmas~\ref{lemma:mixed-block-G2} and~\ref{lemma:vertical-block-G2}
rely on different degree conditions
(cf.~condition~(ix) of Lemma~\ref{lemma:mixed-block-G2}
and condition~(x) of Lemma~\ref{lemma:vertical-block-G2}).
We did not manage to unify them, but we suspect that it may be possible
if one uses some other degree function.
\end{remark}

\section{Proof of the main theorem}
\label{section:proof}

In this section we prove Theorem~\ref{theorem: main}
using preliminary computations
performed in Sections~\ref{section:horizontal},
\ref{section:mixed} and~\ref{section:vertical}.

\begin{proof}[{Proof of Theorem~\ref{theorem: main}}]
Define an auxiliary
triplet $(\widetilde{\QQ}_0,\widetilde{V}_0,\widetilde{R}_0)$
as follows. Put $\widetilde{\QQ}_0=\QQ_0$ and
$$
\widetilde{V}_0=
\{\widetilde{a}_{i,1}\mid i\in [1,k]\}\cup
\{\widetilde{a}_{i,2}\mid i\in [1,k-1]\}\cup
\{a\}.
$$
Define
$$\widetilde{R}_0(k,2)=1,
\quad \widetilde{R}_0(0,1)=\widetilde{R}_0(k,3)=a,$$
and $\widetilde{R}_0(i,j)=\widetilde{a}_{i,j}$ for
$i\in [1,k]$, $j\in [1,2]$, $(i,j)\neq (k,2)$.
Let
$$\widetilde{\psi}_0\colon \TT(\widetilde{V}_0)\to \TT(V_0)$$
be a monomial change of variables given by
$$
\widetilde{a}_{i,j}=\frac{a_{i,j}}{a_{k,n}},\quad a=\frac{1}{a_{k,n}}.
$$
It is easy to check that
$$
\widetilde{\psi}_0^*(F_{\QQ_0,V_0,R_0})=
F_{\widetilde{\QQ}_0,\widetilde{V}_0,\widetilde{R}_0}.
$$

We choose the blocks $B_1,\ldots,B_l$ in the following way.

If $\sum d_i\le k$, then we consecutively choose
$B_1,\ldots,B_l$ to be horizontal blocks of
size~\mbox{$d_1,\ldots, d_l$} situated as high as possible.

If $\sum d_i=k+1$ and $d_l\ge 2$, then we consecutively choose
$B_1,\ldots,B_{l-1}$ to be horizontal blocks of size $d_1,\ldots, d_{l-1}$
situated as high as possible. After this
we choose $B_l$ to be a mixed block of size $d_l$, so that $B_l$
covers all the remaining vertical arrows of $\Ar(\QQ_0)$, and all
horizontal arrows of $\Ar(\QQ_0)$ except for
the arrow~\mbox{$\arrow{(k,2)}{(k,3)}$}.
In particular, if~\mbox{$l=1$} and $d_1=k+1$, then we choose $B_1$ to be the mixed block
that consists of all arrows of~\mbox{$\Ar(\QQ_0)$} except for
the arrow $\arrow{(k,2)}{(k,3)}$.

Finally, if $\sum d_i=k+1$ and $d_l=1$, then we choose
$B_1,\ldots,B_{l-1}$ to be horizontal blocks of size $d_1,\ldots, d_{l-1}$
situated as high as possible in the quiver $\QQ_0$. This means that the
union~\mbox{$B_1\cup\ldots\cup B_{l-1}$}
covers all vertical arrows of $\Ar(\QQ_0)$. After this we choose
$B_l$ to be the first basic vertical block.

In other words, we always choose
$B_1, \ldots, B_l$ so that the blocks $B_i$ and $B_j$ are disjoint
for any $i\neq j$, $i,j\in [1,l]$, and for any $i\in [1,l]$
the quiver $\QQ_i$ with
$\Ver(\QQ_i)=\Ver(\QQ_0)$ and
$$\Ar(\QQ_i)=\Ar(\QQ_0)\setminus (B_1\cup\ldots\cup B_i)$$
is admissible.
Note also that
the arrow $\arrow{(k,2)}{(k,3)}$ is not contained in any
of the blocks $B_i$.

We proceed to define the rational maps $\psi_i$.
If $B_1$ is a mixed block define $\widetilde{\psi}_{1}$
by Lemma~\ref{lemma:mixed-block-G2-start}
and put $\psi_1=\widetilde{\psi}_{1}\circ \widetilde{\psi}_{0}$;
in this case $\psi_1$ is the only change of variables we need.

If $B_1$ is a horizontal block define $\widetilde{\psi}_{1}$
by Lemma~\ref{lemma:horizontal-block-G2-start}.
Put $\psi_1=\widetilde{\psi}_{1}\circ \widetilde{\psi}_{0}$.
We put~\mbox{$\WW_1=\varnothing$} and $\gamma_1=1$ if $d_1=1$, and we put
$\WW_1=\{w\}$ and $\gamma_1=w+1$, where
$a_{w,1}$ is the weight variable used in the proof of
Lemma~\ref{lemma:horizontal-block-G2-start}, if $d_1>1$;
we also put $\MM_1=\varnothing$.
Note that~\mbox{$(\MM_1,\WW_1,\gamma_1)$}
is a block history of $(\QQ_1,\gamma_1)$.
Then we consider the remaining horizontal blocks $B_i$
one by one and define changes of variables $\psi_i$
and block histories~\mbox{$(\MM_i,\WW_i,\gamma_i)$}
by Lemmas~\ref{lemma:horizontal-block-G2}
or~\ref{lemma:horizontal-block-G2-narrow}, depending on whether
the size of $B_i$ exceeds~$1$ or equals~$1$. Note that due
to our choice of the blocks we will first have to
apply Lemma~\ref{lemma:horizontal-block-G2} several times,
and then Lemma~\ref{lemma:horizontal-block-G2-narrow} several times.
If $\sum d_i\le k$,
then this is all we need. If $\sum d_i=k+1$, then
we conclude with a construction of $\psi_l$
applying Lemma~\ref{lemma:mixed-block-G2} or
Lemma~\ref{lemma:vertical-block-G2}, depending on whether the block $B_l$
is mixed or vertical.

Our final observation is that in the process described above we can always
perform the next required step due to compatibility of conditions and
assertions of
Lemmas~\ref{lemma:horizontal-block-G2-start},~\ref{lemma:horizontal-block-G2},~\ref{lemma:mixed-block-G2}
and~\ref{lemma:vertical-block-G2} pointed out in
Remark~\ref{remark:lemmas}.
\end{proof}

\section{Periods}
\label{section: periods for Grassmannians}

In this section we
check that Givental's integral gives the so called main period
for complete intersections in projective spaces and Grassmannians of planes.
To do this we start from an integral of the form of the left hand side of~\eqref{eq:restricted integral} over an
indefinite cycle~\mbox{$\delta_1$} (that we will specify later).
Then we 
take residues
several times
obtaining integrals over cycles $\delta_i$ such that~\mbox{$\delta_{i-1}$} is a boundary of a tubular neighborhood of $\delta_i$.
After taking all residues we define a cycle we integrate over and define all
other cycles one by one. It turns out that the cycle~\mbox{$\delta_1$} we recover in this way is homologous
to a standard cycle $\delta_1^0$ we used in~\eqref{eq:restricted integral}.

\begin{definition}
Let $f$ be a Laurent polynomial in $m$ variables $x_1,\ldots,x_m$.
Let~\mbox{$\Omega(x_1,\ldots,x_m)$} be a standard logarithmic form defined in~\eqref{eq: standard form}.
The integral
$$
I_f(t)=
\int\limits_{|x_i|=\varepsilon_i}\frac{
\Omega(x_1,\ldots,x_m)
}{1-tf} = \sum_{j=0}^\infty t^j \cdot \int\limits_{|x_i|=\varepsilon_i}
f^j\Omega(x_1,\ldots,x_m)
 \in\C[[t]]
$$
is called \emph{the main period} for $f$, where $\varepsilon_i$ are arbitrary positive numbers.
\end{definition}

\begin{remark}
\label{remark: Picard--Fuchs}
Let~$\phi_j$ be the constant term of~$f^j$.
Then $I_f(t)=\sum \phi_j t^j$.
\end{remark}

The following theorem (which is a mathematical folklore, see~\cite[Proposition 2.3]{Prz08} or~\cite[Theorem~3.2]{CCGGK12} for the proof)
justifies this definition.

\begin{theorem}
\label{theorem: Picard--Fuchs}
Let $f$ be a Laurent polynomial in $m$ variables.
Let $P$ be a Picard--Fuchs differential operator for a pencil of hypersurfaces in a torus
provided by $f$.
Then
one has~\mbox{$P[I_f(t)]=0$}.
\end{theorem}

Consider a smooth complete intersection $Y\subset \P^{N}$ of hypersurfaces of degrees $d_1,\ldots,d_l$.
Denote
$$d_0=N+1-\sum d_i.$$
Assume that $d_0\ge 1$, that is $Y$ is a Fano variety.
Anticanonical Givental's Landau--Ginzburg model for $Y$ is given in a torus
$$(\CC^*)^{N}\cong \Spec \C [a_{i,j}^{\pm 1}, y_s^{\pm 1}], \quad i\in[1,l], j\in [1,d_i], s\in [1,d_0-1],$$
by equations
\begin{equation}
\label{projspace}
a_{i,1}+\ldots+a_{i,d_i}=1, \quad i\in [1,l],
\end{equation}
with superpotential
$$w=y_1+\ldots+y_{d_0-1}+\frac{1}{\prod a_{i,j}\prod y_i}.$$

The subvariety cut out by equations~\eqref{projspace}
after change of variables given by
$$x_{i,j}=\frac{a_{i,j}}{a_{i,d_i}},\quad i\in [1,l], j\in [1,d_i-1]$$
is birational to a torus
$$(\CC^*)^m\cong \Spec \C[x_{i,j}^{\pm 1}, y_s^{\pm 1}], \quad i\in[1,l], j\in [1,d_i-1], s\in [1,d_0-1],$$
where $m=N-l$.
The superpotential $w$ gives a Laurent polynomial
$$
f_{Y}=\frac{\prod_{i=1}^l(x_{i,1}+\ldots+x_{i,d_i-1}+1)^{d_i}}{\prod_{i=1}^l \prod_{j=1}^{d_i-1} x_{i,j}\prod_{j=1}^{d_0-1} y_j}+y_1+\ldots+y_{d_0-1}
$$
which is a toric Landau--Ginzburg model for $Y$, see~\cite[\S3.2]{Prz13} and~\cite{ILP13}.

\begin{proposition}
\label{proposition:periods-for-proj-space}
One has
$$I_Y=\int\limits_{
\substack{|x_{i,j}|=\varepsilon_{i,j}\\ |y_s|=\varepsilon_s}} \frac{\Omega(x_{1,1}, \ldots, x_{l,d_l-1}, y_1,\ldots, y_{d_0-1})}{1-tf_Y}.$$
\end{proposition}

\begin{proof}
Consider an integral
$$
I=
\int_{\delta_1}
\frac{
\Omega(a_{1,1},\ldots,a_{l,d_l}, y_1,\ldots, y_{d_0-1})
}
{\prod_{i=1}^l\left(1-\left(a_{i,1}+\ldots+a_{i,d_i}\right)\right)\cdot
\left(1-t\cdot\left(\frac{1}{\prod a_{i,j}\prod y_s}+\sum y_s\right)\right)}
$$
for some $N$-cycle $\delta_1$, cf.~\eqref{eq:restricted integral}.

Put
$$x_{i,j}=\frac{a_{i,j}}{a_{i,d_i}},\quad i\in [1,l], j\in [1,d_i-1].$$
Then 
one has
$$
I=
\int_{\delta_1'}
\frac{\pm
\Omega (x_{1,1},\ldots, x_{1,d_1-1},\ldots, x_{l,1},\ldots, x_{l,d_l-1}, a_{1,d_1},\ldots, a_{l, d_l}, y_1,\ldots, y_{d_0-1})
}{\prod_{i=1}^l\left(1-\left(\sum_{j=1}^{d_i-1} x_{i,j}+1\right)\cdot
a_{i,d_i}\right)\cdot \left(1-t\cdot\left(
\frac{1}{\prod x_{i,j}\prod a_{i,d_i}^{d_i}\prod y_s}+\sum y_s\right)\right)}
$$
for some $N$-cycle $\delta_1'$.

Finally put $$Q_i=1-\left(\sum_{j=1}^{d_i-1} x_{i,j}+1\right)\cdot
a_{i,d_i},\quad i\in [1,l],$$
so that
$$
a_{i,d_i}=\frac{1-Q_i}{\sum_{j=1}^{d_i-1} x_{i,j}+1}.
$$
After this we have
$$
I=
\int_{\delta_1''}
\frac{
\pm\Omega(x_{1,1}, \ldots, x_{l,d_l-1}, Q_1,\ldots, Q_l, y_1,\ldots, y_{d_0-1})
}
{\prod_{i=1}^l \left(1-Q_i\right)\cdot \left(1-t\cdot\left(\frac{\prod_{i=1}^l(x_{i,1}+\ldots+x_{i,d_i-1}+1)^{d_i}}{\prod_{i=1}^l (1-Q_i)^l\prod_{j=1}^{d_i-1} x_{i,j}\prod_{s=1}^{d_0-1} y_s}+y_1+\ldots+y_{d_0-1}\right)\right)}
$$
for some $N$-cycle $\delta_1''$.
Taking residues with respect to variables $Q_i$,
possibly reordering and renaming variables one gets
$$
I=\int_{\Delta} \frac{\Omega(x_{1,1}, \ldots, x_{l,d_l-1}, y_1,\ldots, y_{d_0-1})}{1-tf_Y}
$$
for some $m$-cycle $\Delta$.

Put $\Delta=\{|x_{i,j}|=\varepsilon_{i,j},|y_s|=\varepsilon_s\}$
and define cycles $\delta_2,\ldots,\delta_{l+1}=\Delta$ so that $\delta_{i-1}$ is a boundary of a tubular neighborhood of $\delta_i$ for $i\in [3,l+1]$, and $\delta_1''$ is a boundary of a tubular neighborhood of $\delta_2$.
One can check that $\delta_1''$, and thus also $\delta_1$ is homologous to a cycle $$\delta_1^0=\{|a_{i,j}|=\varepsilon_{i,j},|y_s|=\varepsilon_s\}$$
which completes the proof.
\end{proof}

Now we check that Givental's integral gives the main period for complete intersections in Grassmannians of planes as well.

\begin{proposition}
\label{proposition: periods}
Consider a smooth Fano complete intersection
$$
Y=\G(2,k+2)\cap Y_1 \cap\ldots\cap Y_l
$$
and a nef-partition corresponding to the choice of blocks from the proof of
Theorem~\ref{theorem: main}.
Let~$I^0_Y$ be Givental's integral for this nef-partition, and $\widehat f_Y=F_{\QQ_l,V_l,R_l}$ be a Laurent polynomial in $m=2k-l$ variables
$x_1,\ldots, x_m$ given by Theorem~\ref{theorem: main}. Put $f_Y=\widehat f_Y-l$.
Then
$$
I^0_Y=\int\limits_{|x_i|=\varepsilon_i} \frac{\Omega(x_1,\ldots,x_m)}{1-tf_Y}.
$$
\end{proposition}

\begin{proof}
We choose blocks $B_1,\ldots,B_l$ as in the proof of Theorem~\ref{theorem: main} and put
$$B_0=\Ar(\QQ_0)\setminus \left(\cup_{i\in [1,l]} B_i\right).$$
Note that one has $f_Y=\psi^*F_{\QQ_0,V_0,R_0,B_0}$ where $\psi\colon \TT(V_l)\dasharrow \TT(V_0)$
is the change of variables given by Theorem~\ref{theorem: main}.
Consider an integral
$$
I=\int_{\delta_1}\frac{\Omega(a_{i,j})}{\prod_{s=1}^l (1-F_{\QQ_0,V_0,R_0,B_i})\cdot(1-tF_{\QQ_0,V_0,R_0,B_0})}
$$
over an indefinite $2k$-cycle $\delta_1$, cf.~\eqref{eq:restricted integral}.
Applying the monomial change variables described in the beginning of the proof of Theorem~\ref{theorem: main}
we obtain an integral
$$
I=\int_{\delta_1'}\frac{\pm \Omega(a,\widetilde a_{i,j})}{\prod_{s=1}^l (1-F_{\widetilde \QQ_0,\widetilde
V_0,\widetilde R_0,B_i})\cdot(1-tF_{\widetilde \QQ_0,\widetilde V_0,\widetilde R_0,B_0})}
$$
for some $2k$-cycle $\delta_1'$.

We follow changes of variables from the proof of
Theorem~\ref{theorem: main}.
Consider a form
$$
\Omega=\frac{\Omega(a_1,\ldots,a_p)\cdot F(a_1,\ldots,a_p)}{1-F_{\QQ,V,R,B}},
$$
where $F_{\QQ,V,R,B}$ depends on some variables $a_i$
and a function $ F(a_1,\ldots,a_p)$ is chosen so that
$$
I=\int_{\delta_j}
\frac{\Omega}{1-tf}$$
for some Laurent polynomial $f$ and some $p$-cycle $\delta_j$.
Denote $1-F_{\QQ,V,R,B}$ by $U$.
Depending on whether the block $B$ is a horizontal block containing the arrow $\arrow{(0,1)}{(1,1)}$,
a horizontal block of size at least $2$ not containing the arrow~\mbox{$\arrow{(0,1)}{(1,1)}$},
a basic horizontal block not containing the arrow~\mbox{$\arrow{(0,1)}{(1,1)}$}, a mixed block containing the arrow~\mbox{$\arrow{(0,1)}{(1,1)}$},
a mixed block not containing the arrow~\mbox{$\arrow{(0,1)}{(1,1)}$}, or the first basic vertical block,
we follow changes of variables described in Lemmas~\ref{lemma:horizontal-block-G2-start},~\ref{lemma:horizontal-block-G2},~\ref{%
lemma:horizontal-block-G2-narrow},~\ref{lemma:mixed-block-G2-start},~\ref{lemma:mixed-block-G2}, or~\ref{lemma:vertical-block-G2} respectively.

Suppose that we are not in the situation described in Lemma~\ref{lemma:horizontal-block-G2-narrow}. Let $a_1$ be a main variable and $a_2$ be a weight variable
for the change of variables for
the change of variables that agrees with~\mbox{$(\QQ,V,R)$} and~$B$.
The latter can be decomposed in several changes of variables.
The first change of variables is monomial
so by equation~\eqref{equation: integral changes}
a standard logarithmic form in new variables is equal to a standard logarithmic form in the initial variables $a_i$.
If we are in the situation described in Lemma~\ref{lemma:horizontal-block-G2-narrow}, then we do not change variables on this step
and choose $a_2$ to be the variable $a_{r,1}$ in the notation of Lemma~\ref{lemma:horizontal-block-G2-narrow}.

Keeping the same notation for changed variables for simplicity the form $\Omega$ can be written down as
$$
\frac{da_1}{a_1}\wedge \frac{da_2}{a_2}\wedge\frac{\Omega(a_3,\ldots,a_p)\cdot F(a_1,\ldots,a_p)}{U}.
$$
We put
$$
a_1=\frac{T}{(1-U)\cdot a_2-S}
$$
for certain Laurent polynomials $S$ and $T$ in $a_3,\ldots,a_p$,
cf.~\eqref{eq:main-variable-2-horizontal-start-U}.

After the this substitution the form $\Omega$ is
\begin{multline*}
\frac{da_1}{a_1}\wedge \frac{da_2}{a_2}\wedge\frac{\Omega(a_3,\ldots,a_p)\cdot F(a_1,\ldots,a_p)}{U}=\\
=\big((1-U)\cdot a_2-S\big) \cdot d\left(\frac{1}{(1-U)\cdot a_2-S}\right)\wedge \frac{da_2}{a_2}\wedge
\frac{\Omega(a_3,\ldots,a_p)\cdot F\left(\frac{1}{(1-U)a_2-S},a_2,\ldots,a_p\right)}{U}=\\
=\frac{-1}{(1-U)\cdot a_2-S}\cdot\frac{dU}{U}\wedge da_2\wedge\Omega(a_3,\ldots,a_p)\cdot F\left(\frac{1}{(1-U)a_2-S},a_2,\ldots,a_p\right).
\end{multline*}
After taking a residue with respect to $U$ we get
$$
\Res_U\Omega=\frac{-1}{a_2-S}\cdot da_2\wedge\Omega(a_3,\ldots,a_p)\cdot F\left(\frac{1}{a_2-S},a_2,\ldots,a_p\right).
$$

We put
$$
a_2=R\cdot b+S
$$
for some Laurent polynomial $R$ in $a_3,\ldots,a_p$,
cf.~\eqref{eq:weight-variable-change-horizontal-start}.
Now our new variables are $b, a_3,\ldots,a_p$.
After this substitution we get
\begin{multline*}
\Res_U\Omega=\frac{-1}{a_2-S} \cdot da_2\wedge\Omega(a_3,\ldots,a_p)\cdot F\left(\frac{1}{a_2-S},a_2,\ldots,a_p\right)=\\
=\frac{-1}{R\cdot b} \cdot d\left(R\cdot b+S\right)\wedge\Omega(a_3,\ldots,a_p)\cdot F\left(\frac{1}{R\cdot b},R\cdot b+S,a_3\ldots,a_p\right)=\\
=-\frac{db}{b}\wedge\Omega(a_3,\ldots,a_p)\cdot F\left(\frac{1}{R\cdot b},R\cdot b+S,a_3\ldots,a_p\right).
\end{multline*}
Thus
$$
I=\int_{\delta_j}
\frac{\Omega}{1-tf} =
\int_{\delta_{j+1}}
-\Omega(b,a_3,\ldots,a_p)\cdot \overline{F}\left(b,a_3\ldots,a_p\right)
$$
for some $(p-1)$-cycle $\delta_{j+1}$, where
$$
\overline{F}\left(b,a_3\ldots,a_p\right)=F\left(\frac{1}{R\cdot b},R\cdot b+S,a_3\ldots,a_p\right).
$$

Applying this procedure step by step $l$ times following the proof of Theorem~\ref{theorem: main},
we define cycles $\delta_2,\ldots,\delta_{l+1}=\Delta$ so that $\delta_{i-1}$ is a boundary of a tubular neighborhood of $\delta_i$ for $i\in [3,l+1]$, and $\delta_1'$ is a boundary of a tubular neighborhood of $\delta_2$
and arrive to an integral
$$
\int\limits_{\Delta} \frac{\Omega(x_1,\ldots,x_m)}{1-tf_Y}
$$
for some Laurent polynomial $f_Y$ in some variables $x_1,\ldots,x_m$.
Put $\Delta=\{|x_{i}|=\varepsilon_{i}\}$
and recover the cycles  $\delta_1',\delta_2,\ldots,\delta_{l}$.
One can check that $\delta_1'$, and thus also $\delta_1$ is homologous to a cycle $$\delta_1^0=\{|a_{i,j}|=\varepsilon_{i,j}\}$$
which completes the proof.
\end{proof}

\begin{corollary}
\label{corollary:main}
The proof of Theorem~\ref{theorem: main} provides weak Landau--Ginzburg models for complete intersections in Grassmannians of planes.
\end{corollary}

\begin{proof}
Let $Y$ be a complete intersection in a Grassmannian of planes and let $f_Y$
be a Laurent polynomial given by Theorem~\ref{theorem: main}. In other words a family of hypersurfaces
in a torus corresponding to $f_Y$ is relatively birational to anticanonical Givental's Landau--Ginzburg model for~$Y$. By~\cite{BCFKS98}
and~\cite[Proposition 3.5]{BCK03}
Givental's integral for~$(Y,\omega_Y)$, where $\omega_Y$ is an anticanonical form, equals $\widetilde{I}^Y$. On the other hand, by
Remark~\ref{remark: Picard--Fuchs} it is a constant terms series of $f_Y$,
i.\,e. a main period of~$Y$.
\end{proof}

\section{Hyperplane sections}
\label{section:hyperplane-sections}

In this section we
apply Theorem~\ref{theorem: main}
to obtain explicit formulas
for Laurent polynomials corresponding
to Fano varieties that are sections of Grassmannians of planes
by several hyperplanes.
We will use notation introduced in
Theorem~\ref{theorem: main}.
Keeping in mind Remark~\ref{remark:shift} and Proposition~\ref{proposition: periods},
we will be more interested in the shifted Laurent polynomials
$F_{\QQ_l,V_l,R_l}-l$ than $F_{\QQ_l,V_l,R_l}$ themselves.

\begin{lemma}
\label{lemma:hyperplane-sections}
Suppose that $n=2$, $l\le k$ and $d_1=\ldots=d_l=1$.
Consider the triplet~\mbox{$(\QQ_0,V_0,R_0)$}.
Let $B_i$, $i\in [1,l]$, be the $(i-1)$-th basic horizontal block.
Then there is a sequence of
triplets $(\QQ_i,V_i,R_i)$, $i\in [1,l]$,
and a sequence of changes of variables
$$\psi_i\colon \TT(V_i)\dasharrow \TT(V_{i-1}),\quad i\in [1,l],$$
such that the change of variables
$\psi_i$ agrees with the triplet $(\QQ_{i-1},V_{i-1},R_{i-1})$
and the block~$B_i$, the triplet $(\QQ_i,V_i,R_i)$ is a transformation
of the triplet
$(\QQ_{i-1},V_{i-1},R_{i-1})$ associated to $\psi_i$,
one has
$$
V_l=\{a_{i,1}\mid i\in [1,k]\}\cup
\{a_{i,2}\mid i\in [l,k-1]\},
$$
and the following
conditions hold:
\begin{itemize}
\item[(i)]
the quiver $\QQ_l$ does not
contain vertical arrows $\alpha$ such that $h(\alpha)=(i,j)$ for
$i\in [1,l]$, $j\in [1,2]$;
\item[(ii)]
one has $R(k,2)=1$;
\item[(iii)]
for $(i,j)$ with $i\in [l, k]$,
$j\in [1,2]$, $(i,j)\neq (k,2)$,
one has $R_l(i,j)=a_{i,j}$;
\item[(iv)] for any $i\in [1,l-1]$ one has
$$R_l(i,1)=a_{l,1}+a_{l-1,1}+\ldots+a_{i,1};$$
\item[(v)] for any $i\in [1,l-1]$ one has
$$R_l(i,2)=\frac{a_{l,2}\cdot
(a_{l,1}+a_{l-1,1})\cdot
(a_{l,1}+a_{l-1,1}+a_{l-2,1})\cdot\ldots\cdot
(a_{l,1}+a_{l-1,1}+\ldots+a_{i,1})}
{a_{l-1,1}\cdot a_{l-2,1}\cdot\ldots\cdot a_{i,1}}$$
if $l<k$, and
$$R_k(i,2)=\frac{
(a_{k,1}+a_{k-1,1})\cdot
(a_{k,1}+a_{k-1,1}+a_{k-2,1})\cdot\ldots\cdot
(a_{k,1}+a_{k-1,1}+\ldots+a_{i,1})}
{a_{k-1,1}\cdot a_{k-2,1}\cdot\ldots\cdot a_{i,1}}$$
if $l=k$.
\item[(vi)] one has
$$R_l(k,3)=R(1,1)=a_{l,1}+a_{l-1,1}+\ldots+a_{1,1}.$$
\end{itemize}
\end{lemma}
\begin{proof}
Arguing as in the proof of Theorem~\ref{theorem: main},
we start with the standard triplet~\mbox{$(\QQ_0,V_0,R_0)$} and
obtain a triplet $(\widetilde{\QQ}_0,\widetilde{V}_0,\widetilde{R}_0)$
as described in the proof
of Theorem~\ref{theorem: main}.
Then we apply Lemma~\ref{lemma:horizontal-block-G2-start} with $s=1$
once, and apply Lemma~\ref{lemma:horizontal-block-G2-narrow}
with~\mbox{$\gamma=1$} and~\mbox{$\MM=\WW=\varnothing$}
consecutively $l-1$ times.
\end{proof}

Applying Lemma~\ref{lemma:hyperplane-sections}, we immediately obtain.

\begin{corollary}
\label{corollary:hyperplane-sections}
In the notation of Lemma~\ref{lemma:hyperplane-sections}
suppose that $l\le k-1$. Then one has
\begin{multline}\label{eq:hyperplane-sections}
F_{\QQ_l,V_l,R_l}-l=\\ =
\sum\limits_{i\in [1,l-1]}
\frac{a_{l,2}\cdot (a_{l,1}+a_{l-1,1})\cdot\ldots\cdot
(a_{l,1}+\ldots+a_{i+1,1})}
{a_{l-1,1}\cdot\ldots\cdot a_{i,1}}
+\sum\limits_{i\in [l,k-1]}\frac{a_{i,2}}{a_{i,1}}
+\frac{1}{a_{k,1}}+\\
+\sum\limits_{\substack{i\in [l,k-2]\\ j\in [1,2]}}\frac{a_{i+1,j}}{a_{i,j}}
+\frac{a_{k,1}}{a_{k-1,1}}+\frac{1}{a_{k-1,2}}
+a_{l,1}+a_{l-1,1}+\ldots+a_{1,1}.
\end{multline}
\end{corollary}

\begin{corollary}\label{corollary:index-2}
In the notation of Lemma~\ref{lemma:hyperplane-sections}
suppose that $l=k$. Then one has
\begin{multline}\label{eq:index-2}
F_{\QQ_k,V_k,R_k}-k=\\ =
\sum\limits_{i\in [1,k-1]}
\frac{(a_{k,1}+a_{k-1,1})\cdot\ldots\cdot
(a_{k,1}+\ldots+a_{i+1,1})}
{a_{k-1,1}\cdot\ldots\cdot a_{i,1}}
+\frac{1}{a_{k,1}}
+a_{k,1}+a_{k-1,1}+\ldots+a_{1,1}.
\end{multline}
\end{corollary}

Now we proceed to the case corresponding to a Fano variety
that is a section of the Grassmannian $\G(2,k+2)$ by
$k+1$ hyperplanes.

\begin{lemma}
\label{lemma:hyperplane-sections-index-1}
Suppose that $n=2$, $l=k+1$ and $d_1=\ldots=d_{k+1}=1$.
Consider the triplet~\mbox{$(\QQ_0,V_0,R_0)$}.
Let $B_i$, $i\in [1,k]$, be the $(i-1)$-th basic horizontal block,
and let $B_{k+1}$ be the first basic vertical block.
Then there is a sequence of
triplets $(\QQ_i,V_i,R_i)$, $i\in [1,k+1]$,
and a sequence of changes of variables
$$\psi_i\colon \TT(V_i)\dasharrow \TT(V_{i-1}),\quad i\in [1,k+1],$$
such that the change of variables
$\psi_i$ agrees with the triplet $(\QQ_{i-1},V_{i-1},R_{i-1})$
and the block~$B_i$, the triplet $(\QQ_i,V_i,R_i)$ is a transformation
of the triplet
$(\QQ_{i-1},V_{i-1},R_{i-1})$ associated to $\psi_i$,
one has
$$
V_{k+1}=\{a_{i,1}\mid i\in [2,k]\}
$$
and
\begin{multline}\label{eq:index-1}
F_{\QQ_{k+1},V_{k+1},R_{k+1}}- (k+1)=\\ =
\left(
a_{k,1}\cdot
\frac{
(1+a_{k-1,1})\cdot\ldots\cdot
(1+a_{k-1,1}+\ldots+a_{2,1})}
{a_{k-1,1}\cdot\ldots\cdot a_{2,1}}
+\vphantom{
\sum\limits_{i\in [1,k-1]}
\frac{(1+a_{k-1,1})\cdot\ldots\cdot
(1+\ldots+a_{i+1,1})}
{a_{k-1,1}\cdot\ldots\cdot a_{i,1}}
}
\right. \\ \left. +
\sum\limits_{i\in [2,k-1]}
\frac{(1+a_{k-1,1})\cdot\ldots\cdot
(1+\ldots+a_{i+1,1})}
{a_{k-1,1}\cdot\ldots\cdot a_{i,1}}
+1\right)\times \\ \times
\left(
1+a_{k-1,1}+\ldots+a_{2,1}+\frac{1}{a_{k,1}}
\right).
\end{multline}
\end{lemma}
\begin{proof}
We obtain changes of variables $\psi_1,\ldots,\psi_k$ from
Lemma~\ref{lemma:hyperplane-sections}. Then we make a change
of variables $\psi_{k+1}$ that agrees with the triplet $(\QQ_k,V_k,R_k)$
and the block $B_{k-1}$ applying Lemma~\ref{lemma:vertical-block-G2} with $\gamma=1$ and
$\MM=\WW=\varnothing$.
Equation~\eqref{eq:index-1} follows by direct computation.
\end{proof}

\begin{remark}[{cf. Problem~\ref{problem: CY}}]
\label{remark:compactify}
One can easily see that families of hypersurfaces given by
equations~\eqref{eq:hyperplane-sections},
\eqref{eq:index-2} and~\eqref{eq:index-1}
can be compactified to singular Calabi--Yau hypersurfaces
by multiplying by denominators and homogenizing.
\end{remark}

\begin{problem}[{cf. Problem~\ref{problem:smoothing}}]
Prove that the compactifications mentioned in Remark~\ref{remark:compactify}
admit crepant resolutions.
In other words, prove that these weak Landau--Ginzburg models are
weak ones.
In addition prove that the corresponding
toric varieties admit smoothings to hyperplane sections of Grassmannians,
that is, prove that equations~\eqref{eq:hyperplane-sections},
\eqref{eq:index-2} and~\eqref{eq:index-1} give toric Landau--Ginzburg
models.
\end{problem}

\section{Examples}
\label{section:examples}

In this section we provide several sporadic examples that illustrate our
computations performed in
Sections~\ref{section:horizontal},
\ref{section:mixed} and~\ref{section:vertical}.
We will use notation introduced in
Theorem~\ref{theorem: main}.
Keeping in mind Remark~\ref{remark:shift} and Proposition~\ref{proposition: periods},
we will be more interested in the shifted Laurent polynomials
$F_{\QQ_l,V_l,R_l}-l$ than $F_{\QQ_l,V_l,R_l}$ themselves, just as in Section~\ref{section:hyperplane-sections}.
Weak Landau--Ginzburg models for threefold examples
(coinciding with ours up to monomial changes of variables)
can be found in~\cite{Prz13} and~\cite{CCGGK12}, where they are
obtained by different methods.


\begin{example}\label{example:quadric-threefold}
The following computation corresponds to a quadric threefold,
which we treat as a hyperplane section of the
Grassmannian $\G(2,4)$.

Let $k=n=2$, $l=1$ and $d_1=1$ in the notation of
Theorem~\ref{theorem: main}.
In this case equation~\eqref{eq:hyperplane-sections} gives
$$
F_{\QQ_1,V_1,R_1}-1=
\frac{a_{1,2}}{a_{1,1}}+
\frac{1}{a_{2,1}}+
\frac{a_{2,1}}{a_{1,1}}+
\frac{1}{a_{1,2}}+a_{1,1}.
$$
This polynomial is, up to monomial change of variables,
the toric Landau--Ginzburg model for quadric threefold
written down in~\cite[Example~2.2]{KP12}.
\end{example}

\begin{example}\label{example:G25-two-hyperplanes}
The following computation corresponds to a Fano fourfold
of index $3$ that is a section of the Grassmannian
$\G(2,5)$ by two hyperplanes.

Let $n=2$, $k=3$, $l=2$ and $d_1=d_2=1$.
In this case equation~\eqref{eq:hyperplane-sections} gives
$$
F_{\QQ_2,V_2,R_2}-2=
\frac{a_{2,2}}{a_{1,1}}+
\frac{a_{2,2}}{a_{2,1}}+
\frac{1}{a_{3,1}}+
\frac{a_{3,1}}{a_{2,1}}+
\frac{1}{a_{2,2}}+a_{2,1}+a_{1,1}.
$$
\end{example}

\begin{example}\label{example:quadric-surface}
The following computation corresponds to a smooth quadric surface,
which we treat as an intersection of two hyperplanes in the
Grassmannian $\G(2,4)$. The same result
was known earlier;
actually, it is just a
simplified Givental's Landau--Ginzburg model for the
quadric surface treated as a toric variety~$\mbox{$\P^1\times\P^1$}$,
see Section~\ref{section:toric}.

Let $n=k=2$, $l=2$ and $d_1=d_2=1$.
In this case~\eqref{eq:index-2} gives
$$
F_{\QQ_2,V_2,R_2}-2=
\frac{1}{a_{1,1}}+
\frac{1}{a_{2,1}}+
a_{2,1}+a_{1,1}.
$$
\end{example}

\begin{example}\label{example:quadric-surface-2}
The following example provides another computation that
corresponds to a smooth quadric surface. Similarly to
Example~\ref{example:quadric-surface}, we treat the quadric
surface as an intersection of two hyperplanes in the
Grassmannian $\G(2,4)$, and we use variable changes that agree with
various blocks to obtain the result, but unlike
Example~\ref{example:quadric-surface} (or rather
Lemma~\ref{lemma:hyperplane-sections} where the actual computation is performed)
we do not follow
exactly the procedure prescribed by the proof of
Theorem~\ref{theorem: main}. Our point here is that
our procedure is not the only one, and sometimes not even the shortest
one, to obtain the answer.

We have $n=k=2$.
We start with the standard triplet $(\QQ_0,V_0,R_0)$ and
obtain a triplet $(\widetilde{\QQ}_0,\widetilde{V}_0,\widetilde{R}_0)$
as described in the proof
of Theorem~\ref{theorem: main}.
Then we make variable changes that agree with
the $0$-th horizontal basic block, which consists of a single arrow
$\arrow{(0,1)}{(1,1)}$, and with the second vertical
basic block, which consists of a single arrow
$\arrow{(2,2)}{(2,3)}$. This gives us two equations
$\widetilde{a}=\widetilde{a}_{1,1}$ and $\widetilde{a}=1$,
which we use to exclude variables $\widetilde{a}$ and
$\widetilde{a}_{1,1}$. This gives us a triplet
$(\QQ, V, R)$ such that
$\Ar(\QQ)$ consists of the arrows
$\arrow{(1,1)}{(1,2)}$, $\arrow{(2,1)}{(2,2)}$,
$\arrow{(1,1)}{(2,1)}$ and $\arrow{(1,2)}{(2,2)}$,
one has $V=\{a_{1,2}, a_{2,1}\}$, and
$$
R(0,1)=R(1,1)=R(2,2)=R(2,3)=1, \quad R(1,2)=a_{1,2}, \quad R(2,1)=a_{2,1}.
$$
We compute
$$
F_{\QQ,V,R}-2=
a_{1,2}+a_{2,1}+\frac{1}{a_{1,2}}+\frac{1}{a_{2,1}}.
$$

More than this, one can start from the triplet $(\QQ_0,V_0,R_0)$
itself, and utilize the same two blocks to obtain equations $a_{1,1}=1$
and $a_{2,2}=1$. Using these equations to exclude
variables $a_{1,1}$ and $a_{2,2}$ we obtain exactly the same result
as above.
\end{example}

\begin{example}
The following computation corresponds to a Fano threefold
of anticanonical degree $40$ and index $2$ that is a section of the Grassmannian $\G(2,5)$ by three hyperplanes
(see e.\,g.~\cite[\S3.4]{IsPr99}).
Due to~\cite{Ga07}, this variety has a terminal Gorenstein toric degeneration.
One can see that the Laurent polynomial
that we get is given by a procedure discussed in Section~\ref{section:toric}.
In fact it is a toric Landau--Ginzburg model, see~\cite[Theorem 18]{Prz13}.

Let $n=2$, $k=3$, $l=3$ and $d_1=d_2=d_3=1$.
In this case~\eqref{eq:index-2} gives
$$
F_{\QQ_3,V_3,R_3}-3=
\frac{a_{3,1}+a_{2,1}}
{a_{2,1}\cdot a_{1,1}}+
\frac{1}{a_{2,1}}+
\frac{1}{a_{3,1}}+
a_{3,1}+a_{2,1}+a_{1,1}.
$$
\end{example}

\begin{example}[{cf.~\cite[Example 1.2]{Pr16}}]
\label{example:X14-dim-4}
The following computation corresponds to a Fano fourfold
of index~$2$ that is a section of the Grassmannian $\G(2,6)$
by four hyperplanes.

Let $n=2$, $k=4$, $l=4$ and $d_1=d_2=d_3=d_4=1$.
In this case
equation~\eqref{eq:index-2} gives
\begin{multline*}
F_{\QQ_4,V_4,R_4}-4=\\
=
\frac{(a_{4,1}+a_{3,1})\cdot (a_{4,1}+a_{3,1}+a_{2,1})}
{a_{3,1}\cdot a_{2,1}\cdot a_{1,1}}+
\frac{a_{4,1}+a_{3,1}}
{a_{3,1}\cdot a_{2,1}}+
\frac{1}{a_{3,1}}+
\frac{1}{a_{4,1}}+
a_{4,1}+a_{3,1}+a_{2,1}+a_{1,1}.
\end{multline*}
In~\cite{PSh14} the relative compactification of a family of hypersurfaces in $(\C^*)^4$ given by
this Laurent polynomial is computed. This computation confirms expectations of Homological Mirror Symmetry
in this case.
\end{example}

\begin{example}
The following computation corresponds to a Fano fivefold
of index $2$ that is a section of the Grassmannian $\G(2,7)$
by five hyperplanes.

Let $n=2$, $k=5$, $l=5$ and $d_1=d_2=d_3=d_4=d_5=1$.
In this case~\eqref{eq:index-2} gives
\begin{multline*}
F_{\QQ_5,V_5,R_5}-5=
\frac{(a_{5,1}+a_{4,1})\cdot (a_{5,1}+a_{4,1}+a_{3,1})\cdot
(a_{5,1}+a_{4,1}+a_{3,1}+a_{2,1})}
{a_{4,1}\cdot a_{3,1}\cdot a_{2,1}\cdot a_{1,1}}+\\
+
\frac{(a_{5,1}+a_{4,1})\cdot (a_{5,1}+a_{4,1}+a_{3,1})}
{a_{4,1}\cdot a_{3,1}\cdot a_{2,1}}
+
\frac{a_{5,1}+a_{4,1}}
{a_{4,1}\cdot a_{3,1}}
+\frac{1}{a_{4,1}}+\frac{1}{a_{5,1}}+
\\ +a_{5,1}+a_{4,1}+a_{3,1}+a_{2,1}+a_{1,1}.
\end{multline*}
\end{example}

\begin{example}\label{example:S5}
The following computation corresponds to a
del Pezzo surface of degree~$5$ that is a section of
the Grassmannian $\G(2,5)$ by four hyperplanes.

Let $n=2$, $k=3$, $l=4$ and $d_1=d_2=d_3=d_4=1$.
In this case equation~\eqref{eq:index-1} gives
$$
F_{\QQ_4,V_4,R_4}-4=
\left(
a_{3,1}\cdot\frac{1+a_{2,1}}{a_{2,1}}+
\frac{1}{a_{2,1}}+1
\right)\cdot\left(
1+a_{2,1}+\frac{1}{a_{3,1}}
\right).
$$
\end{example}

\begin{example}\label{example:V14}
The following computation corresponds to a Fano threefold
of anticanonical degree $14$ and
index $1$ that is a section of the Grassmannian $\G(2,6)$
by five hyperplanes (see~\cite[\S12.2]{IsPr99}).

Let $n=2$, $k=4$, $l=5$ and $d_1=d_2=d_3=d_4=d_5=1$.
In this case equation~\eqref{eq:index-1} gives
\begin{multline*}
F_{\QQ_5,V_5,R_5}-5=\\=
\left(
a_{4,1}
\cdot\frac{(1+a_{3,1})\cdot (1+a_{3,1}+a_{2,1})}{a_{3,1}\cdot a_{2,1}}
+\frac{1+a_{3,1}}{a_{3,1}\cdot a_{2,1}}+\frac{1}{a_{3,1}}+1
\right)\times \\ \times \left(
1+a_{2,1}+a_{3,1}+\frac{1}{a_{4,1}}
\right).
\end{multline*}
\end{example}

\begin{example}\label{example:22}
The following computation corresponds to a three-dimensional complete
intersection of two quadrics (one of which we treat as the
Grassmannian $\G(2,4)$).

Let $n=k=2$, $l=1$ and $d_1=2$.
Following the proofs of Theorem~\ref{theorem: main} and Lemma~\ref{lemma:horizontal-block-G2-start},
we arrive to a triplet $(\QQ_1,V_1,R_1)$ such that
$$V_1=\{a_{1,1}, a_{1,2}, a_{2,1}\},$$
and
$$
F_{\QQ_1,V_1,R_1}-1=
a_{1,2}+\frac{1}{a_{2,1}}+
\frac{1}{a_{1,1}}\cdot\left(
a_{1,1}+a_{2,1}+\frac{1}{a_{1,2}}
\right)^2.
$$
\end{example}

\begin{example}\label{example:22-dim2}
The following computation corresponds to a two-dimensional complete
intersection of two quadrics, which we treat as a section of the
Grassmannian $\G(2,4)$ by a quadric and a hypersurface.

Let $n=k=2$, $l=2$, $d_1=2$ and $d_2=1$.
Following the proofs of Theorem~\ref{theorem: main} and Lemmas~\ref{lemma:horizontal-block-G2-start}
and~\ref{lemma:vertical-block-G2}, we
arrive to a triplet $(\QQ_2,V_2,R_2)$ such that
$$V_2=\{a_{1,1}, a_{1,2}\}$$
and
$$
F_{\QQ_2,V_2,R_2}-2=
R_2(2,3)=\left(a_{1,1}+a_{1,2}\right)\cdot\left( 1+\frac{1}{a_{1,1}}+\frac{1}{a_{1,2}}\right)^2.
$$
\end{example}

\begin{example}\label{example:23}
The following computation corresponds to a three-dimensional complete
intersection of a quadric (which we treat as the
Grassmannian $\G(2,4)$) and a cubic.

Let $n=k=2$, $l=1$ and $d_1=3$.
Following the proofs of Theorem~\ref{theorem: main} and Lemma~\ref{lemma:mixed-block-G2-start}, we
arrive to a triplet $(\QQ_1,V_1,R_1)$ such that
$$V_1=\{a_{1,1}, a_{1,2}, a_{2,1}\}$$
and
$$
F_{\QQ_1,V_1,R_1}-1=
R_1(2,3)=\frac{a_{1,1}}{a_{1,2}}\cdot \left(
a_{1,2}+
\frac{a_{2,1}^2+a_{1,1}\cdot a_{2,1}+a_{1,1}+a_{2,1}}
{a_{1,1}\cdot a_{2,1}}
\right)^3.
$$
\end{example}

\begin{example}\label{example:X10-dim-4}
The following computation corresponds to a Fano fourfold
of anticanonical degree $160$ and
index $2$ that is a section of the Grassmannian $\G(2,5)$
by a quadric and a hyperplane
(see e.\,g.~\cite{DIM-fourfold}).

Let $n=2$, $k=3$, $l=2$, $d_1=2$ and $d_2=1$.
Following the proofs of Theorem~\ref{theorem: main} and Lemmas~\ref{lemma:horizontal-block-G2-start}
and~\ref{lemma:horizontal-block-G2-narrow}, we
arrive to a triplet $(\QQ_2,V_2,R_2)$ such that
$$V_2=\{a_{1,1}, a_{1,2}, a_{2,1}, a_{3,1}\},$$
and
$$
F_{\QQ_2,V_2,R_2}-2=
a_{1,2}+\frac{1}{a_{2,1}}+
\frac{1}{a_{3,1}}+
\frac{1}{a_{1,1}}\cdot\left(a_{1,1}+a_{2,1}+a_{3,1}+
\frac{1}{a_{1,2}}+\frac{a_{3,1}}{a_{1,2}\cdot a_{2,1}}\right)^2.
$$
\end{example}

\begin{example}\label{example:V10}
The following computation corresponds to a Fano threefold
of anticanonical degree $10$ and index $1$ that is an intersection of the Grassmannian $\G(2,5)$ with a quadric and two hyperplanes
(see e.\,g.~\cite[\S5.1]{IsPr99}, \cite{DIM-threefold}).

Let $n=2$, $k=3$, $l=3$, $d_1=2$ and $d_2=d_3=1$.
Following the proofs of Theorem~\ref{theorem: main} and Lemmas~\ref{lemma:horizontal-block-G2-start},
\ref{lemma:horizontal-block-G2-narrow} and~\ref{lemma:vertical-block-G2}, we
arrive to a triplet $(\QQ_3,V_3,R_3)$ such that
$$V_3=\{a_{1,1}, a_{1,2}, a_{3,1}\}$$
and
$$
F_{\QQ_3,V_3,R_3}-3=
R_3(3,3)=\frac{a_{3,1}+a_{1,2}+1}{a_{1,1}}\cdot\left(
a_{1,1}+\frac{1}{a_{3,1}}+1+\frac{1}{a_{1,2}}+\frac{a_{3,1}}{a_{1,2}}
\right)^2.
$$
\end{example}

\begin{example}\label{example:X20-dim-4}
The following computation corresponds to a Fano fourfold
of anticanonical degree $20$ and index $1$
that is an intersection of the Grassmannian $\G(2,5)$
with two quadrics.

Let $n=2$, $k=3$, $l=2$ and $d_1=d_2=2$.
Following the proofs of Theorem~\ref{theorem: main} and Lemmas~\ref{lemma:horizontal-block-G2-start},
\ref{lemma:horizontal-block-G2-narrow} and~\ref{lemma:mixed-block-G2}, we
arrive to a triplet $(\QQ_2,V_2,R_2)$ such that
$$V_2=\{a_{1,1}, a_{1,2}, a_{2,2}, a_{3,1}\}$$
and
\begin{multline*}
F_{\QQ_2,V_2,R_2}-2=
R_2(3,3)=\\ =
\frac{1}{a_{1,1}}\cdot\left(
a_{1,1}+\left(\frac{1+a_{2,2}}{a_{3,1}}
+\frac{a_{2,2}}{a_{1,2}}\right)\cdot \left(
a_{3,1}+\frac{1+a_{1,2}\cdot a_{2,2}+a_{2,2}}{a_{2,2}}
\right)^2
\right)^2.
\end{multline*}
\end{example}

\begin{remark}
Changes of variables described in Theorem~\ref{theorem: main} and choices of basic blocks for them
are not unique ones that give Laurent polynomials. However in all cases we know
all these Laurent polynomials for a given variety differ by cluster mutations,
that is they correspond to families of hypersurfaces that are fiberwise birational and have a common (Calabi--Yau)
compactification. It would be interesting to find out if this is true in general?
\end{remark}

\section{Discussion}
\label{section:discussion}

Changes
of variables discussed in sections~\ref{section:horizontal},~\ref{section:mixed},~\ref{section:vertical}, and~\ref{section:proof}
can be made for complete intersections in all Grassmannians $\G(n,k+n)$. Moreover,
in~\cite[Lemma~3.2.2]{BCFKS98} and~\cite[Theorem~3.2.13]{BCFKS98}
the natural generalization of maximal nef-partition
from Grassmannians to partial flag varieties is suggested.
So we expect that our proof of Theorem~\ref{theorem: main} can be generalized to these cases.

\begin{problem}[see~\cite{DH15} and~\cite{PSh15b}]
\label{problem: all Grassmannians}
Following Theorem~\ref{theorem: main} and Proposition~\ref{proposition: periods} show the existence of weak Landau--Ginzburg models
for all complete intersections in Grassmannians and, more generally, partial flag varieties.
\end{problem}

Let us mention that to solve Problem~\ref{problem: all Grassmannians} it is not enough
to represent Givental's type Landau--Ginzburg models by Laurent polynomials.
One should keep track of a particular type of change of variables in the spirit of Proposition~\ref{proposition: periods}
to check that under them the periods are preserved (cf.~\cite[Proposition 4.4]{PSh15b}).
However our experience shows that checking that a period condition is
preserved for explicitly described birational transformations does not cause any difficulties.

In Theorem~\ref{theorem: main} we used a specific nef-partition to construct a Laurent polynomial.
However sometimes one can use another nef-partitions and get a Laurent polynomial as well.
In Example~\ref{example:quadric-surface-2} we use other nef-partition to get the same result
as one gets by Theorem~\ref{theorem: main}. In some examples we consider
one can get, using different nef-partitions, different Laurent polynomials.
However they are mutationally equivalent to ones we get by Theorem~\ref{theorem: main}.

\begin{question}
Is this always the case?
\end{question}

One more motivation for this question is given by~\cite{Li13}; the similar result is announced in~\cite[Theorem 5.1]{CKP14} as
a part of T.\,Prince's Thesis. That is,
there are different methods that, under some assumptions, allow one to obtain Laurent polynomials for Givental's Landau--Ginzburg models.
In~\cite{Li13} and in T.\,Prince's Thesis the resulting Laurent polynomials are proved
to be actually independent on a choice of a nef-partition: Laurent polynomials obtained
from different nef-partitions are relatively birational. In other words, they differ by mutations (cf.~\cite[Example 1.1]{Pr16}).

\smallskip
According to a private communication with A.\,Harder, the families
of hypersurfaces in tori given by Laurent polynomials we obtain in the proof of
Theorem~\ref{theorem: main} have relative compactifications that are general
complete intersections in toric varieties (cf. Remark~\ref{remark:compactify}).
Moreover, they are Calabi--Yau compactifications.
These compactifications enable one to compute the number of components of the
unique reducible fibers of the compactifications.

\begin{problem}
\label{problem: CY}
Prove the existence of Calabi--Yau compactifications of weak Landau--Ginzburg models for complete intersections in Grassmannians of
planes obtained in the proof of Theorem~\ref{theorem: main}. In other words, prove that these models are weak ones.
If Problem~\ref{problem: all Grassmannians} is solved, prove this for complete intersections in arbitrary Grassmannians or, more generally,
partial flag varieties.
\end{problem}

\smallskip
Another natural problem is the following.

\begin{problem}[cf.~\cite{DH15}]\label{problem:smoothing}
Prove that the toric varieties whose fan polytopes are Newton polytopes of weak Landau--Ginzburg models for complete intersections in Grassmannians of planes obtained in the proof of
Theorem~\ref{theorem: main} can be smoothed to
corresponding complete intersections. In other words, prove that these models are toric ones.
If Problem~\ref{problem: all Grassmannians} is solved, prove this for complete intersections in arbitrary Grassmannians or, more generally,
partial flag varieties.
\end{problem}

\smallskip
Problem~\ref{problem: all Grassmannians} can be generalized further.
Many interesting higher-dimensional Fano varieties are not complete intersections in Grassmannians or partial flag manifolds but
sections of non-decomposable vector bundles such as symmetric or skew powers of tautological vector bundle,
see, for instance,~\cite{Kuz15},~\cite{Kuz16}.

\begin{question}
How to describe their analogs of nef-partitions for Grassmannians or partial flag varieties
for smooth Fano varieties that are sections of non-decomposable vector bundles. Does the analog of Theorem~\ref{theorem: main}
holds for them? Can analogs of Problems~\ref{problem: CY} and~\ref{problem:smoothing} be solved for them?
\end{question}

\smallskip
Suppose that Problem~\ref{problem: CY} is solved.
Let $Y$ be a complete intersection of dimension~$r$ in Grassmannian of planes and let $LG(Y)$ be its Calabi--Yau compactification.
Since birational smooth Calabi--Yau varieties are birational in codimension one, the number of irreducible components
in each fiber of $LG(Y)$ does not depend on a particular compactification. It is expected that there is at most one
reducible fiber of $LG(Y)$. Denote the number
of its irreducible components by $\widehat{k}$, and put $k_{LG(Y)}=\widehat{k}-1$.
This number can be computed
via the approach mentioned above
that was communicated to us by A.\,Harder.

\begin{conjecture}[{see~\cite{GKR12} and~\cite[Conjecture 1.1]{PSh13}}]
\label{conjecture:kLG}
Let $r\geqslant 3$. Then
$$
k_{LG(Y)}=h^{1,r-1}(Y).
$$
For $r=2$ one has
$$
k_{LG(Y)}=h^{1,1}(Y)-1.
$$
\end{conjecture}

\begin{problem}
Prove Conjecture~\ref{conjecture:kLG} for weak Landau--Ginzburg
models for complete intersections in Grassmannians of planes obtained in the proof of Theorem~\ref{theorem: main}.
If Problem~\ref{problem: all Grassmannians} is solved, prove this for complete intersections in arbitrary Grassmannians or, more generally,
partial flag varieties.
\end{problem}


\begin{thebibliography}{ACGH}

\bibitem[ATY85]{ATY85}
L.\,Aizenberg, A.\,Tsikh, A.\,Yuzhakov, \emph{Higher-dimensional residues and their applications Current problems in mathematics}, Encyclopaedia of Mathematical Sciences, Vol. 8,  5--64, 1985.


\bibitem[Ba97]{Ba97}
V.\,V.\,Batyrev, {\it Toric Degenerations of Fano Varieties and
Constructing Mirror Manifolds}, Collino, Alberto (ed.) et al.,
The Fano conference. Papers of the conference, organized to commemorate
the 50th anniversary of the death of Gino Fano (1871--1952),
Torino, Italy, September 29--October 5, 2002.
Torino: Universita di Torino, Dipartimento di Matematica. 109--122 (2004). 

\bibitem[BCFKS97]{BCFKS97}
V.\,V.\,Batyrev, I.\,Ciocan-Fontanine, B.\,Kim, D.\,van Straten,
{\it Conifold transitions and mirror symmetry for Calabi--Yau complete intersections in Grassmannians},
Nucl. Phys., B 514, No. 3, 640--666 (1998). 

\bibitem[BCFKS98]{BCFKS98}
V.\,V.\,Batyrev, I.\,Ciocan-Fontanine, B.\,Kim, D.\,van Straten,
{\it Mirror Symmetry and Toric Degenerations of Partial Flag
Manifolds}, Acta Math. 184, No. 1 (2000), 1--39. 

\bibitem[BCFK03]{BCK03}
A.\,Bertram, I.\,Ciocan-Fontanine, B.\,Kim, {\it Two Proofs of a
Conjecture of Hori and Vafa}, Duke Math. J. 126, No. 1, 101--136 (2005). 

\bibitem[CCGGK12]{CCGGK12}
T.\,Coates, A.\,Corti, S.\,Galkin, V.\,Golyshev, A.\,Kasprzyk, {\it Mirror Symmetry and Fano Manifolds}, European Congress of Mathematics
(Krakow, 2--7 July, 2012), November 2013, pp. 285--300.

\bibitem[CKP14]{CKP14}
T.\,Coates, A.\,Kasprzyk, T.\,Prince,
{\it Four-dimensional Fano toric complete intersections}, 
Proc. R. Soc. A {471}:20140704 (2015).


\bibitem[CLS11]{CLS11}
D.\,Cox, J.\,Little, H.\,Schenck, {\it Toric varieties},
Graduate Studies in Mathematics 124. Providence, RI: AMS (2011).

\bibitem[DH15]{DH15}
C.\,Doran, A.\,Harder, \emph{Toric Degenerations and the Laurent polynomials related to Givental's Landau--Ginzburg models},
Canad. J. Math. 68 (2016), no. 4, 784--815.


\bibitem[DHKLP]{DHKLP}
{C.\,Doran, A.\,Harder, L.\,Katzarkov, J.\,Lewis, V.\,Przyjalkowski},
\emph{Modularity of Fano threefolds}, in preparation.

\bibitem[DIM12]{DIM-threefold}
O.\,Debarre, A.\,Iliev, L.\,Manivel,
\emph{On the period map for prime Fano threefolds of degree 10}, J. Algebr. Geom. 21, No. 1, 21--59 (2012).

\bibitem[DIM12]{DIM-fourfold}
O.\,Debarre, A.\,Iliev, L.\, Manivel,
\emph{Special prime Fano fourfolds of degree 10 and index 2}, 
Recent advances in algebraic geometry. A volume in honor of Rob
Lazarsfeld’s 60th birthday. Cambridge:
CUP. LMS Lecture Note Series {417} (2014), 123--155.


\bibitem[EHX97]{EHX97} T.\,Eguchi, K.\,Hori, C.-Sh.\,Xiong,
\emph{Gravitational quantum cohomology}, Int. J. Mod. Phys. A 12,
No. 9, 1743--1782 (1997). 

\bibitem[Ga07]{Ga07}
S.\,Galkin, \emph{Small toric degenerations of Fano 3-folds}, preprint \url{http://www.mi.ras.ru/~galkin/work/3a.pdf}.

\bibitem[Gi97a]{Gi97a}
A.\,Givental, {\it Stationary phase integrals, quantum Toda lattices, flag manifolds and the mirror conjecture}, Topics in singularity theory,  103--115, Amer. Math. Soc. Transl. Ser. 2, 180, AMS, Providence, RI, 1997.

\bibitem[Gi97b]{Gi97b}
A.\,Givental, {\it A mirror theorem for toric complete intersections}, Topological field theory, primitive forms and related topics (Kyoto, 1996),  141--175, Progr. Math., 160, Birkhauser Boston, Boston, MA, 1998.

\bibitem[GKR12]{GKR12}
M.\,Gross, L.\,Katzarkov, H.\,Ruddat, {\it Towards Mirror Symmetry for Varieties of General Type}, 
Adv. Math. \textbf{308} (2017), 208--275.

\bibitem[HV00]{HV00}
K.\,Hori, C.\,Vafa, {\it Mirror symmetry}, arXiv:hep-th/0002222.

\bibitem[ILP13]{ILP13}
N.\,Ilten, J.\,Lewis, V.\,Przyjalkowski. \emph{Toric Degenerations of Fano Threefolds Giving Weak Landau--Ginzburg Models},
Journal of Algebra 374 (2013), 104--121.

\bibitem[IP99]{IsPr99}
{ V.\,Iskovskikh, Yu.\,Prokhorov}, \emph{Fano varieties},
Encyclopaedia of Mathematical Sciences, \textbf{47} (1999) Springer, Berlin.

\bibitem[KP12]{KP12}
L.\,Katzarkov, V.\,Przyjalkowski,
\newblock \emph{Landau--Ginzburg models --- old and new.}
\newblock Akbulut, Selman (ed.) et al., Proceedings of the 18th Gokova geometry--topology conference.
Somerville, MA: International Press; Gokova: Gokova Geometry-Topology Conferences, 97--124 (2012).

\bibitem[Ki00]{Ki00}
B.\,Kim, \emph{Quantum hyperplane section principle for concavex decomposable vector bundles},
J. Korean Math. Soc. 37 (2000), no. 3, 455--461.

\bibitem[Ko94]{Ko94}
M.\,Kontsevich, {\it Homological algebra of mirror symmetry}, Proc.
International Congress of Matematicians (Z\"{u}rich 1994),
Birkh\"{a}uzer, Basel, 1995, pp. 120--139. 

\bibitem[Kuz15]{Kuz15}
A.\,Kuznetsov,
\emph{On K\"uchle varieties with Picard number greater than 1},
Izvestiya: Mathematics, \textbf{79}:4 (2015), 698--709.

\bibitem[Kuz16]{Kuz16}
A.\,Kuznetsov,
\emph{K\"uchle fivefolds of type $c_5$}, Math. Z., 284:3 (2016), 1245--1278.

\bibitem[Lee01]{Lee01}
Y.\,Lee, \emph{Quantum Lefschetz hyperplane theorem}, Invent. Math.  145  (2001),  no. 1, 121--149.

\bibitem[Li13]{Li13}
Zh.\,Li, \emph{On the birationality of complete intersections associated to nef-partitions},
arXiv:1310.2310.

\bibitem[Ma99]{Ma99}
Yu.\,Manin, {\it Frobenius manifolds, quantum cohomology, and moduli
spaces}, Colloquium Publications. American Mathematical Society
(AMS). 47. Providence, RI: American Mathematical Society (AMS)
(1999).

\bibitem[MR13]{MR13}
R.\,Marsh, K.\,Rietsch, {\it The B-model connection and mirror symmetry for Grassmannians}, arXiv:1307.1085.

\bibitem[Pr16]{Pr16}
T.\,Prince, \emph{Efficiently computing torus charts in Landau--Ginzburg models of complete intersections in Grassmannians of planes},
to appear in Bull. of the KMS.

\bibitem[Prz08]{Prz08}
V.\,Przyjalkowski, \emph{On Landau--Ginzburg models for Fano
varieties}, Comm. Num. Th. Phys., Vol. 1, No. 4, 713--728 (2008).

\bibitem[Prz10]{Prz10}
V.\,Przyjalkowski, \emph{Hori--Vafa mirror models for complete
intersections in weighted projective spaces and weak
Landau--Ginzburg models}, Cent. Eur. J. Math. 9, No. 5, 972--977 (2011). 

\bibitem[Prz13]{Prz13}
V.\,Przyjalkowski, \emph{Weak Landau--Ginzburg models for smooth Fano threefolds}, Izv. Math. Vol., 77 No. 4 (2013), 135--160.

\bibitem[Prz16]{Prz16}
V.\,Przyjalkowski, \emph{Calabi--Yau compactifications of toric Landau--Ginzburg models for smooth Fano threefolds}, Sbornik: Mathematics, 2017, 208:7, 992--1013. 

\bibitem[Prz18]{Prz18}
V.\,Przyalkowski, \emph{On Calabi--Yau compactifications of toric Landau--Ginzburg models for Fano complete intersections}, Mathematical Notes, 102 (2018),
arXiv:1701.08532.


\bibitem[PSh14a]{PSh-arxiv}
V.\,Przyjalkowski, C.\,Shramov,
{\it Laurent phenomenon for Landau--Ginzburg models of complete intersections in Grassmannians of planes},
arXiv:1409.3729.


\bibitem[PSh14b]{PSh14}
V.\,Przyjalkowski, C.\,Shramov, {\it On weak Landau--Ginzburg models
for complete intersections in Grassmannians},
Russian Math. Surveys 69, No. 6, 1129--1131 (2014).

\bibitem[PSh15a]{PSh13}
V.\,Przyjalkowski, C.\,Shramov, {\it On Hodge numbers of complete intersections and Landau--Ginzburg models},
Int. Math. Res. Not. IMRN, 2015:21 (2015), 11302--11332.

\bibitem[PSh15b]{PSh15b}
V.\,Przyjalkowski, C.\,Shramov, {\it Laurent phenomenon for Landau--Ginzburg models of complete intersections in Grassmannians},
Proc. Steklov Inst. Math., 290 (2015), 91--102.

\bibitem[St93]{St93}
B.\,Sturmfels, {\it Algorithms in Invariant Theory}, Texts and Monographs in Symbolic
Computation, Wien, Springer--Verlag, 1993.

\end{thebibliography}
\end{document}